\documentclass[final]{siamart1116}
\usepackage{amsmath,amsfonts,amssymb}  
\usepackage{url}  
\usepackage{epstopdf}
\usepackage{graphicx}
\graphicspath{{./Figures/}}
\usepackage{pifont}
\usepackage{pstricks, pst-node, pst-plot}
\usepackage{epsfig}
\usepackage{sidecap}
\usepackage{colortbl}
\usepackage{color}
\usepackage{booktabs}
\usepackage[inline]{enumitem}
\usepackage{tasks}

\newcommand{\trace}[1]{\textrm{tra}\,#1}

\newcommand{\vectorr}[1]{\mathbf{#1}}

\title{Multiple Local and Global Bifurcations and Their Role in Quorum Sensing Dynamics}

\author{
	Mariana Harris\thanks{Department of Computational Medicine, UCLA, Los Angeles, CA90095, USA.
	\email{mharris94@g.ucla.edu}.}
	\and
	Viviana~Rivera--Estay\thanks{Departamento de Matem{\'a}tica, F{\'i}sica y Estad{\'i}stica, Universidad Cat{\'o}lica del Maule, San Miguel 3605, Talca, Chile.
	\email{vivianarivera.mate@gmail.com}. 
	Funded by Agencia Nacional de Investigac\'on y Desarrollo de Chile (Grant: Beca Doctorado Nacional N$^\circ 21211263$).}
	\and
	Pablo Aguirre\thanks{Departamento de Matem{\'a}tica, Universidad T{\'e}cnica Federico Santa Mar{\'i}a, Casilla 110-V, Valpara{\'i}so, Chile.
	\email{pablo.aguirre@usm.cl}. 
	Partially funded by Proyecto Interno UTFSM PI$\_$LIR$\_$24$\_$04.}
	\and
	V\'ictor F.~Bre\~na--Medina\thanks{Department of Mathematics, ITAM, R\'io Hondo 1, Ciudad de M\'exico 01080, M\'exico.
	\email{victor.brena@itam.mx}.
	Funded by Asociación Mexicana de Cultura A.C.}
}

\begin{document}
\maketitle

\begin{abstract}
Quorum sensing governs bacterial communication, playing a crucial role in regulating population behaviour. We propose a mathematical model that uncovers chaotic dynamics within quorum sensing networks, highlighting challenges to predictability. The model explores interactions between autoinducers and two bacterial subtypes, revealing oscillatory dynamics in both a constant autoinducer sub-model and the full three-component model. In the latter case, we find that the complicated dynamics can be explained by the presence of homoclinic Shilnikov bifurcations. We employed a combination of normal form analysis and numerical continuation methods to analyse the system. 
\end{abstract}  



\pagestyle{myheadings}
\thispagestyle{plain}
\markboth{M.~Harris, V.~Rivera--Estay, P.~Aguirre and V.~Bre\~na--Medina}
{Multiple Local and Global Bifurcations and Their Role in Quorum Sensing Dynamics}

\section{Introduction}
\label{sec:intro}

Quorum sensing (QS) ---also known as auto-induction--- is a mechanism of regulation of gene expression which allows communication in a cell population~\cite{miller01,waters05}. In the case of bacteria, this process is a density dependent behaviour, which consists of cell-to-cell communication through the release of chemical signal molecules known as {\em autoinducers}. This communication enables a population to explore a given medium and alter its behaviour in response to its own fluctuation, which occurs from changes in the number of species present in the community. Specifically, the basis of the QS mechanism is the production and release of autoinducers, whose concentration increases as a function of bacterial population density. When the population density exceeds a threshold, the bacteria detect the autoinducers which then activate genes that switch on a behavioural trait. Autoinducers located in the external medium diffuse and bind with a specific protein within the bacteria. A protein-autoinducer complex is then formed that attaches to a region of the bacteria's DNA. This in turn regulates the production of the autoinducers enhancing a specific regulated cell-density behaviour; see, for instance,~\cite{chopp01,dockery01,ward01}.

The QS phenomenon was first observed by Nealson and  Hastings in 1970 in a bioluminescent bacteria called \textit{Vibrio fischeri}~\cite{article:Nelson1970}. This bacteria forms a mutually beneficial symbiotic relationship with some species of squid such as the Hawaiian squid \textit{Euprymna scolopes}. The \textit{V. fischeri} live in the squid's light organ which provides nourishment, allowing the bacteria to proliferate. When the bacteria reach a high population density the genes involved in bioluminescence are expressed. The produced light allows the squid to mask its shadow and avoid predation~\cite{article:Visick2000}. Following the discovery of the \textit{V. fisheri}'s density-dependent bioluminescent behaviour, other species of bacteria have been discovered to employ a QS mechanism; for instance, the bacteria \textit{Chromobacterium violeceum} gives rise to the production of a purple pigment~\cite{article:McClean1997}. Also, biofilms may present challenges  and opportunities in the industry setting as these may cause blockages in specialised machinery~\cite{Wirtanen1992}, gene expression~\cite{article:Lewis2002}, { inhibiting pathogen virulence~\cite{article:Hentzer2003, article:Wu2004}, and disrupting QS signalling~\cite{ article:Bjarnsholt2007,article:Yi2007}.  That is, quorum sensing inhibitors function by blocking signal receptors or degrading autoinducers, thereby altering bacterial gene expression and population dynamics~\cite{jiang,remy}. These interventions have been shown to reduce biofilm formation and disrupt survival pathways, indirectly increasing decay rates~\cite{naga2}. Furthermore although the specific strategy of combining QS inhibitors with periodic modulation of the decay rate remains an area for future research, periodically varying bacterial decay rates could provide an additional level of control by introducing fluctuating conditions that prevent bacterial adaptation to constant stress. This strategy may exploit bacterial dynamics by periodically altering environmental factors, such as nutrient availability or physical conditions, which can exacerbate the vulnerabilities of bacterial populations and inhibit resistance development. Alternating between higher and lower decay rates could disrupt bacterial survival strategies, leading to a more effective control mechanism over time. Such temporal variability in decay rates could enhance the efficacy of QS inhibitors by making bacterial populations more susceptible during periods of transition, preventing long-term adaptation and resistance. That is, upon integrating QS inhibitors into bacterial control strategies with periodic decay rate modulation, their ability to modulate microbial behaviours can be exploited, advancing applications in antimicrobial therapies and bioprocess management.
}

The QS phenomena has been studied from either deterministic as well as stochastic and hybrid mathematical modelling approaches, which are inspired by many biologically plausible perspectives. For instance, in~\cite{judithpv} a plant-pathogen population controlling virulence on leaf surfaces is investigated by means of a continuous-time Markov process, where a linear birth and a logistic-death-migration population processes along with an autocatalytic mechanism for {\em acyl homoserine lactone} autoinducers concentration are taken into account. Their findings include that QS mechanism and autoinducers diffusion follow an inversely proportional relationship. Another interesting work can be seen in~\cite{frederick}; there, extracellular polymeric substances~(EPS) production is modelled in a growing biofilm under various environmental conditions, which yields to a reaction-diffusion system where the diffusion coefficient is density-dependent. The authors found that QS-induced EPS production permits a biofilm to switch from a colonisation state to a protection state, which is a crucial characteristic of QS. Even further, in~\cite{kuttler} an anomalous diffusion process is taken into account, which captures a delayed maundering of substances as a consequence of transcriptional features. On the other hand, a deterministic gene regulatory network point of view is addressed in~\cite{chen,elowitz,ojalvo} by considering a QS pathway involving multiple feedforward and negative feedback loops aside from transcription time delays in a cancer drug realising scheme. A key result from their manuscripts lie on the appearance of Hopf bifurcations, other secondary oscillatory-induced bifurcations and a time-delay threshold, which coordinate self-sustained oscillating features. This small sample of approaches pays special attention to transcriptional time-delay as well as on-and-off activating crucial switches. Our area of focus is primarily directed towards the later. 

In this work we propose a simple model which consists of features of a bacterial population and its inherent autoinducer concentration interaction. In other words, we assume that the population of bacteria is locally activated and inhibited by the production of autoinducers.
In order to capture the key qualitative ingredients of QS mechanisms in bacteria, we follow a simplified approach. In so doing, we assume that the autoinducers bind to receptors that enhance the expression of a particular gene, and its production also directly depends on the bacteria population~(e.~g.~\cite{article:Zi2012}). When a bacterium responds to an external stimulus by increasing the amount of a cellular component, the process is {\em up-regulated}, and it is {\em down-regulated} otherwise. To put it another way, the dynamics of the bacteria are regulated by both, autoinducers as well as growth of the bacteria population itself. Hence, the population is considered to be composed of cell sub-populations in two different states: up-regulated and down-regulated bacteria, namely motile~$m$ and static~$s$, respectively, and the concentration of autoinducers is given by~$q$. As autoinducers are produced by both sub-populations at a rate $r$ (see~\cite{article:Waters2005}), and the bacteria can either switch or remain in their category, we assume that:
\begin{enumerate*}[label=(\roman*)]
\item the growth rate of the motile bacteria population follows an autocatalytic process and is proportional to the probability that motile bacteria switch on-and-off in their current category. The response rate of motile bacteria to autoinducer concentration is then given by $k_1$ and is up-regulated by its own production and down-regulated by the production of static bacteria; that is, the dynamics follows an activator-inhibitor-like behaviour, where $m$ acts as the activator and $s$  as the inhibitor. We also assume that the media may be overpopulated, so $\gamma$ is a saturation parameter;
\item the production of static bacteria is assumed as a result of a cross-catalytic reaction of motile bacteria; this interaction may be taken as an indirect down-regulation process, which is modulated by $k_2$;
\item the entire bacteria population is constantly produced at a rate $(1+\varepsilon)\alpha$, where $0< \varepsilon <1$ is the ratio between both bacteria sub-populations growth rates. The autoinducers, motile and static bacteria decay at rates $\mu_1$, $\mu_2$ and $\mu_3$, respectively.  
\end{enumerate*}
Putting everything together, we obtain that the dynamical local interaction is governed by the system 
\begin{gather}
X:\left\{\begin{array}{rcl} \label{eq:3d}
\dot q &=& r(m+s)-\mu_1q\,,\\
\dot m & = & \dfrac{k_1qm^2}{s(1+\gamma m^2)}+\alpha-\mu_2m\,, \\
\dot s & = & k_2qm^2+\epsilon\alpha-\mu_3s\,.
\end{array}\right.
\end{gather}

Essentially, model \eqref{eq:3d} captures the dynamical behaviour of~QS, namely, the concentration of autoinducers depends on the production by bacteria in both states. When this concentration lies beyond a certain threshold, bacteria are expected to react in a synchronised-like way, as has been reported in~\cite{kaur,Li,taylor}, for instance. We have identified this feature by means of slowly varying a crucial parameter as can be seen further in section~\ref{sec:constantq}, and analyse the different dynamical scenarios of \eqref{eq:3d} to find parameter regimes for steady state long term behaviour and the observed periodic oscillations. However, we also find chaos which can be explained by the presence of Shilnikov homoclinic bifurcations.
 
The present manuscript is organised as follows: throughout section~\ref{sec:s0} we prove that the model solutions are always non-negative for feasible initial conditions. In section~\ref{sec:constantq} we take into consideration a constant autoinducer concentration and study the local stability of a reduced system. This is followed in  Section~\ref{sec:bifanal} by a bifurcation analysis where the role of a key parameter triggers leading nonlinear events, namely Hopf and Bogdanov--Takens bifurcations. Normal forms for both dynamical phenomena are also computed. From there, a thorough numerical bifurcation analysis is performed in section~\ref{sec:bifanal3D} in the full system, where conditions to the Shilnikov homoclinic chaos mechanism are found and numerically explored. Concluding remarks can be found in section~\ref{sec:disc}.   

\section{Non-negative solutions for realistic initial conditions}
\label{sec:s0}

Model (\ref{eq:3d}) is well-posed in the sense that every solution starting from a realistic (positive) initial condition remains non-negative.  In what follows we see (\ref{eq:3d}) as a vector field defined in the set $\Omega:=\{(q,m,s)\in\mathbb{R}^3|\,\,q\geq0, m\geq0, s>0\}$ and use the following standard notation~\cite{Arnold,Du06} for its components: $X(q,m,s)=X_1(q,m,s)\dfrac{\partial}{\partial q}+X_2(q,m,s)\dfrac{\partial}{\partial m}+X_3(q,m,s)\dfrac{\partial}{\partial s}$.

Let us study the behaviour of  $X$ in $\partial\Omega$ to show that $\Omega$ is invariant.
Let us subdivide $\partial\Omega$ into the following coordinate planes: $\partial\Omega=\omega_{qm}\cup\omega_{qs}\cup\omega_{ms}$, where $\omega_{qm}=\{(q,m,s)\in\mathbb{R}^3|\,\,q\geq0, m\geq0, s=0\}$, $\omega_{qs}=\{(q,m,s)\in\mathbb{R}^3|\,\,q\geq0, m=0, s>0\}$, and $\omega_{ms}=\{(q,m,s)\in\mathbb{R}^3|\,\,q=0, m\geq0, s>0\}$.

The restriction of (\ref{eq:3d}) to the plane $\omega_{qs}$ is $X(q,0,s)=(rs-\mu_1 q)\dfrac{\partial}{\partial q}+\alpha\dfrac{\partial}{\partial m}{+(\epsilon\alpha+\mu_3 s)\dfrac{\partial}{\partial s}}$. Since $\alpha>0$, the vector field (\ref{eq:3d}) on $\omega_{qs}$ points towards the interior of $\Omega$. Similarly, the restriction of (\ref{eq:3d}) to the plane $\omega_{ms}$ is $X(0,m,s)=r(m+s)\dfrac{\partial}{\partial q}+(\alpha-\mu_2 m)\dfrac{\partial}{\partial m}{+(\epsilon\alpha+\mu_3 s)\dfrac{\partial}{\partial s}}$. Since $r>0$, it follows that the vector field $X$ in $\omega_{ms}$ points towards the interior of $\Omega$.

On the other hand, system (\ref{eq:3d}) is not defined in the plane $\omega_{qm}$. In order to analyse the behaviour of (\ref{eq:3d}) near $\omega_{qm}$, let us consider the time scaling $t\mapsto st$. In this way, (\ref{eq:3d}) becomes
\begin{equation}
X_0:\left\{\begin{array}{rcl} \label{eq:s0}
\dot q &=& r(m+s)s-\mu_1qs,\\
\dot m & = & \dfrac{k_1qm^2}{1+\gamma m^2}+\alpha s-\mu_2ms, \\
\dot s & = & k_2qm^2s+\epsilon\alpha s-\mu_3s^2.
\end{array}\right.
\end{equation}
System (\ref{eq:s0}) is topologically equivalent to (\ref{eq:3d}) in $\Omega$. Moreover, (\ref{eq:s0}) is well defined for $s=0$ and, hence, it can be continually extended to the boundary plane $\omega_{qm}\subset\partial\Omega$. The restriction of (\ref{eq:s0}) to the plane $\omega_{qm}$ is $X_0(q,m,0)=\dfrac{k_1qm^2}{1+\gamma m^2}\dfrac{\partial}{\partial m}$. It follows that the (unique) solution of (\ref{eq:s0}) with initial condition $(q(0),m(0),s(0))=(q_0,m_0,0)\in\omega_{qm}$ is given by
\begin{equation}
\left\{\begin{array}{rcl} \label{eq:rectas}
q(t) &=& q_0,\\
m(t) & = & \dfrac{-1+k_1q_0m_0t+\gamma m_0^2+\sqrt{4\gamma m_0^2+(k_1tq_0m_0+\gamma m_0^2-1)^2}}{2\gamma m_0},\\
s(t) & = & 0.
\end{array}\right.
\end{equation}
As a consequence, the set $\omega_{qm}$ is an invariant plane of (\ref{eq:s0}) which consists of a continuum of straight lines 
parallel to the $m$-axis parameterized by (\ref{eq:rectas}). Moreover, both axes $q=0$ and $m=0$ consist of a continuum of equilibria.

Let us now study the behaviour of orbits of (\ref{eq:s0}) near the invariant plane $\omega_{qm}$. Let us search for the solution of (\ref{eq:s0}) with initial condition $(q(0),m(0),s(0))=(q_0,m_0,\delta)\in{\rm int}(\Omega)$, with $0<\delta\ll1$ sufficiently small. Specifically, consider a solution of the form:
\begin{equation}
\left\{\begin{array}{rcl} \label{eq:pertu1}
q(t) &=& q_0(t)+\delta q_1(t)+O(\delta^2),\\
m(t) & = & m_0(t)+\delta m_1(t)+O(\delta^2),\\
s(t) & = & s_0(t)+\delta s_1(t)+O(\delta^2),
\end{array}\right.
\end{equation}
where $(q_0(t),m_0(t),s_0(t))$ is the solution of (\ref{eq:s0}) in the limit as $\delta\rightarrow0$ and, hence, it is given by (\ref{eq:rectas}). In this way, (\ref{eq:pertu1}) is expressed as
\begin{equation}
\left\{\begin{array}{rcl} \label{eq:pertu2}
q(t) &=& q_0+\delta q_1(t)+O(\delta^2),\\
m(t) & = & m_0(t)+\delta m_1(t)+O(\delta^2),\\
s(t) & = & \delta s_1(t)+O(\delta^2),
\end{array}\right.
\end{equation}
with $m_0(t)= \dfrac{-1+k_1q_0m_0t+\gamma m_0^2+\sqrt{4\gamma m_0^2+(k_1tq_0m_0+\gamma m_0^2-1)^2}}{2\gamma m_0}.$ It follows that the higher order terms of (\ref{eq:pertu2}) must satisfy the initial conditions $(q_1(0),m_1(0),s_1(0))=(0,0,1)$ for $n=1$, and $(q_n(0),m_n(0),s_n(0))=(0,0,0)$ for $n\geq2$. 
Substitution of (\ref{eq:pertu2}) into (\ref{eq:s0}) leads to the following differential equations for the $O(\delta$) terms of $q(t)$:
\begin{flalign}
\delta\dot{q}_1(t) = &\; r\left(m_0(t)+\delta m_1(t)+\ldots+\delta s_1(t)+\ldots\right)\left(\delta s_1(t)+\ldots\right)\nonumber\\
&-\mu_1\left(q_0+\delta q_1(t)+\ldots\right)\left(\delta s_1(t)+\ldots\right) =\delta rm_0(t)s_1(t)-\delta\mu_1 q_0s_1(t)+O(\delta^2).\nonumber
\end{flalign}
Hence,
\begin{subequations}\label{eq:qms1}
\begin{equation}\label{eq:q1}
\dot{q}_1(t)=rm_0(t)s_1(t)-\delta\mu_1 q_0s_1(t), \hspace{5mm} q_1(0)=0.
\end{equation}

Similarly,
$$\dot m(t)=\dot{m}_0(t)+\delta\dot{m}_1(t) +O(\delta^2) = \dfrac{k_1q(t)m^2(t)}{1+\gamma m^2(t)}+\alpha s(t)-\mu_2 m(t)s(t).$$
For $0<\delta\ll1$ sufficiently small we have:
\begin{flalign}
\dot{m}_0(t)+\delta\dot{m}_1(t) +O(\delta^2) =& \;\dfrac{k_1q_0m_0^2(t)}{1+\gamma m_0^2(t)} \nonumber \\
&+\delta\left( \dfrac{k_1q_1(t)m_0^2(t)}{1+\gamma m_0^2(t)}- \dfrac{2\gamma k_1q_0m_0^3(t)m_1(t)}{(1+\gamma m_0^2(t))^2}+ \dfrac{2k_1q_0m_0^2(t)m_1(t)}{1+\gamma m_0^2(t)}\right)
 \nonumber\\
& + \delta\alpha s_1(t)-\delta\mu_2 m_0(t)s_1(t)+O(\delta^2).\nonumber
\end{flalign}
Thus,
\begin{flalign}\label{eq:m1}
\dot{m}_1(t)=&\:\dfrac{k_1q_1(t)m_0^2(t)}{1+\gamma m_0^2(t)}- \dfrac{2\gamma k_1q_0m_0^3(t)m_1(t)}{(1+\gamma m_0^2(t))^2}+ \dfrac{2k_1q_0m_0^2(t)m_1(t)}{1+\gamma m_0^2(t)}+ \alpha s_1(t)-\mu_2 m_0(t)s_1(t), \, \nonumber\\
&\:m_1(0)=0.
\end{flalign}
Likewise, 
\begin{flalign}
\delta\dot{s}_1(t) =& \;k_2\left(q_0+\delta q_1(t)+\ldots\right)\left(m_0(t)+\delta m_1(t)+\ldots\right)^2\left(\delta s_1(t)+\ldots\right)+\epsilon\alpha\left(\delta s_1(t)+\ldots\right)\nonumber\\
& - \mu_3\left(\delta s_1(t)+\ldots\right)^2 = \delta k_2 q_0m_0^2(t)s_1(t)+\delta\epsilon\alpha s_1(t)+O(\delta^2).\nonumber
\end{flalign}
Hence,
\begin{equation}\label{eq:s1}
\dot{s}_1(t)=\left(k_2 q_0m_0^2(t)+\epsilon\alpha\right)s_1(t), \hspace{5mm} s_1(0)=1.
\end{equation}
\end{subequations}
The initial value problem~\eqref{eq:qms1} defines any solution of (\ref{eq:s0}) starting at a distance $0<\delta\ll1$ from the invariant plane $\omega_{qm}$ with accuracy  $O(1/\delta)$. In particular, 
since $k_2 q_0m_0^2(t)+\epsilon\alpha>0$, it follows from~\eqref{eq:s1} that $s_1(t)$ is an increasing function for every $t>0$. Hence, the component $s(t)$ in (\ref{eq:s0}) is increasing near the plane $s=0$. Therefore, no trajectory of (\ref{eq:s0}) with initial condition in ${\rm int}(\Omega)$ can reach the boundary plane $\omega_{qm}$ in (finite or infinite) positive time.

Since (\ref{eq:s0}) is topologically equivalent to \eqref{eq:3d} in $\Omega$, we conclude from the analysis in this section that every trajectory of the original system \eqref{eq:3d} starting in $\Omega$ remains non-negative.

\section{Constant autoinducer concentration}
\label{sec:constantq}

To shed light on the understanding of QS we now analyse the bacteria interaction dynamics under the assumption of a constant autoinducer concentration $q_0 > 0$. In so doing, we have $\eta_1 = k_1q_0$ and $\eta_2 = k_2q_0$ as the motile and static bacteria response rates respectively. From~\eqref{eq:3d}, upon substituting the rescaled variables for the bacteria population~$u= (\eta_2/\eta_1)m$ and $ v= (\mu_3\eta_2/\eta_1^2)s$ as well as $\tau = \mu_3 t$ for time, and the new parameters
\begin{gather}
K = \gamma \left(\dfrac{\eta_1}{\eta_2}\right)^2\,, \quad  b = \dfrac{\mu_2}{\mu_3}\,, \quad  a= \dfrac{\eta_2 \alpha}{\eta_1 \mu_3}\,, \quad  e = \dfrac{\mu_3 \varepsilon}{\eta_1},\, \label{eq:4.10c}
\end{gather}
we obtain the following 2D constant autoinducer system  
\begin{equation}
\left\{\begin{array}{rl}\label{eq:4.10}
\dot u =f(u,v),  & f(u,v)=\dfrac{u^2}{v(1+Ku^2)}+a-bu, \\
\dot v = g(u,v), &  g(u,v)=u^2+ae-v.
\end{array}\right.
\end{equation}

Notice that key parameters arise so that: 
\begin{enumerate*}[label=(\roman*)]
\item $K$ holds a saturation role, 
\item $b$ characterises the decaying bacteria rate, and 
\item $a$ and $e$ capture bacteria production-related roles as well as inversely and directly proportional dependence of the constant autoinducers' concentration and decaying rate variations, respectively. The latter shows that parameter product $ae$ only depends on the bacteria response to autoinducers capacities and production of static subpopulation.
\end{enumerate*}

\begin{figure}[t!]
    \centering
    \includegraphics[scale=.24]{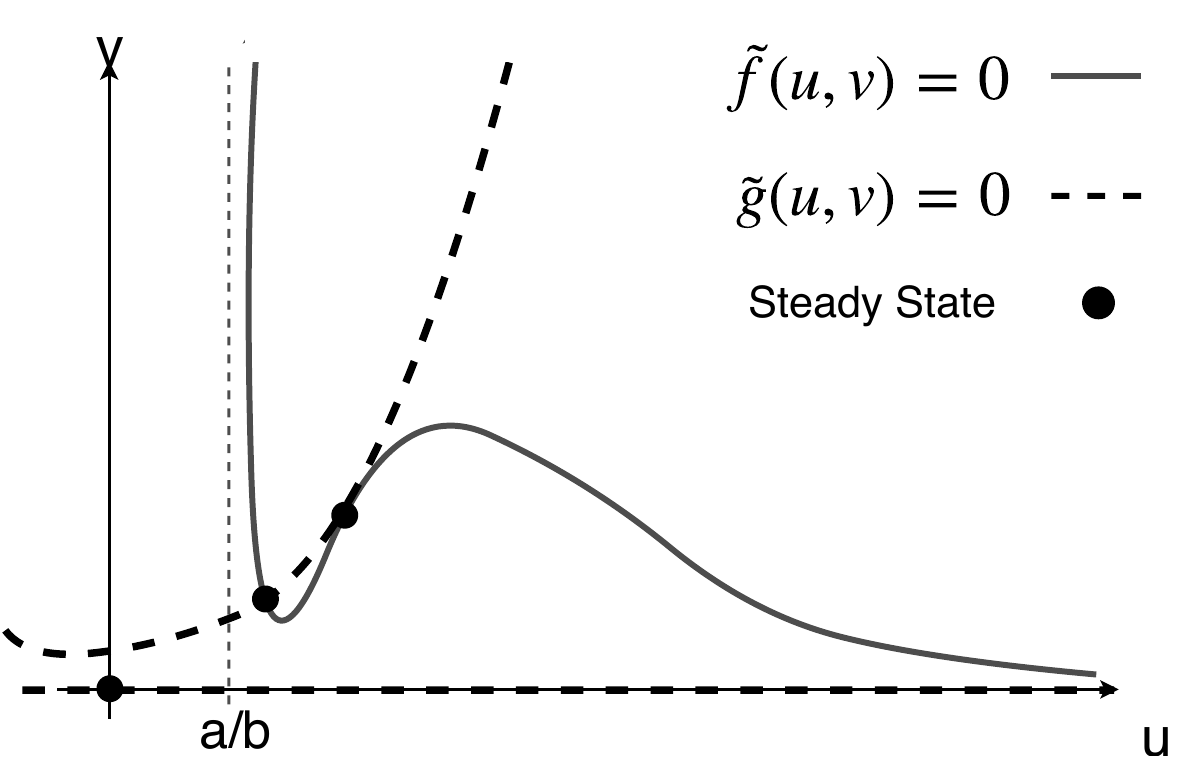}%
    \qquad
    \includegraphics[scale=.24]{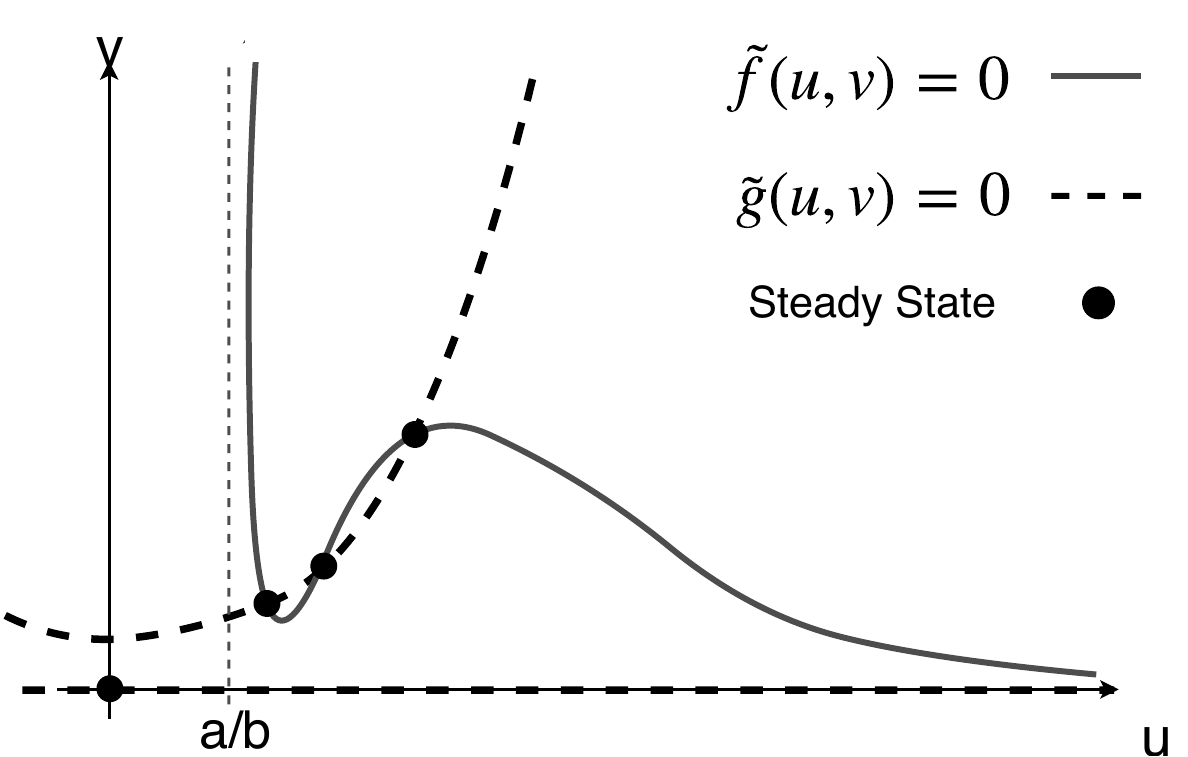}\\
    $(a)$ \hspace{4cm} $(b)$\\
    \includegraphics[scale=.24]{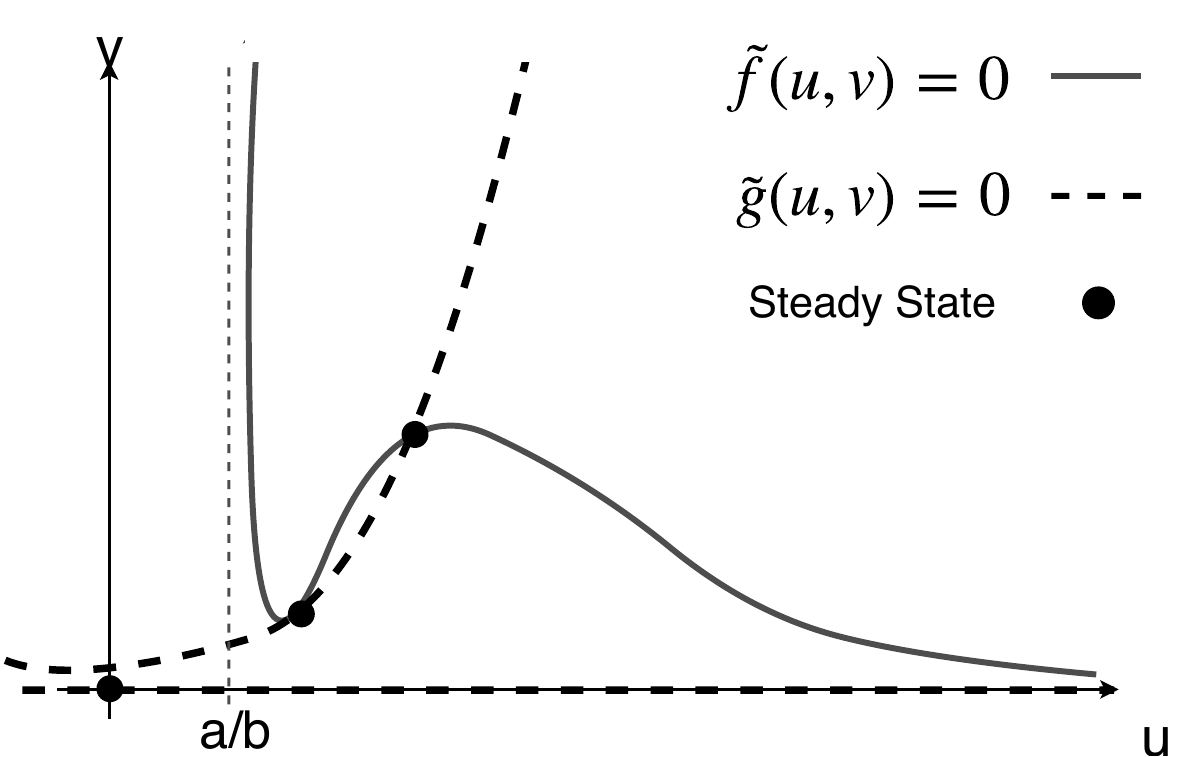}%
    \qquad
    \includegraphics[scale=.24]{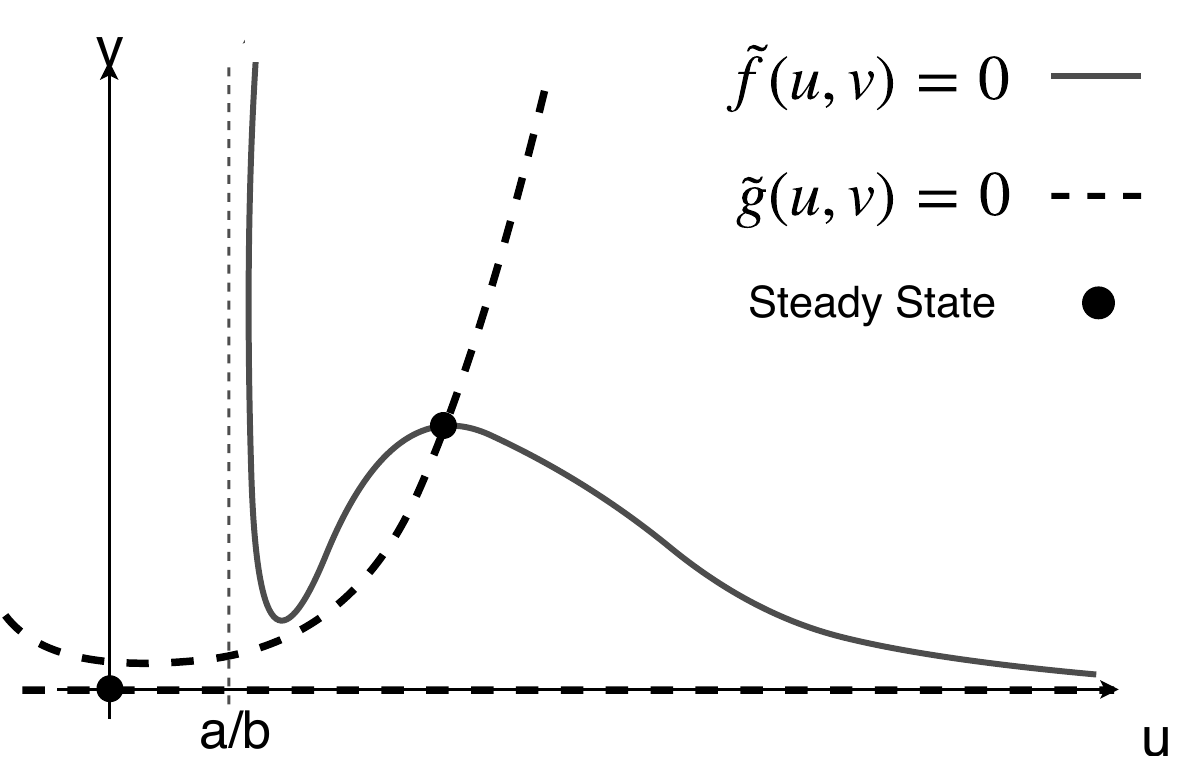}\\
    $(c)$ \hspace{4cm} $(d)$
\caption{Nullclines sketches for key parameter {values of~$e$, where $b>K>a$}. Panels (a) and (c) show three steady-states, where two are positive (within first quadrant) and the third one is the origin; in panel (b) three positive steady-states are depicted; and panel (d) shows only one positive steady-state. Panels (a) and (c) correspond to scenarios approximately at two saddle-node bifurcations.
Here $\tilde{f}(u,v)=0$ and $\tilde{g}(u,v)=0$ stand for $\dot{u}=0$ and $\dot{v}=0$ in \eqref{eq:4.11}, respectively.}
\label{Fig:4.2}
\end{figure}

\subsection{Number of steady-states} 
\label{sec:SteadySt1}

As the understanding of interactions between motile and static bacteria is the backbone of communication mechanism via autoinducers, we seek the circumstances affecting the way subpopulations coexist or go extinct. Hence, we first look for positive steady-states. Consider that nullclines of~\eqref{eq:4.10} satisfy
\begin{subequations}\label{eq:4.13}
\begin{gather}
v = \dfrac{u^2}{(1+Ku^2)(bu-a)}\,, \label{eq:4.13b}\\
v = u^2 +ae \,, \label{eq:4.13a}
\end{gather}
\end{subequations}
for $u,v>0$. Notice that~\eqref{eq:4.13a} defines a convex parabola with vertex at $(0,ae)$, and~\eqref{eq:4.13b} has an asymptote at $u = a /b $. {For values of $u< a/b$ , the corresponding $v$-values are negative, which is not biologically meaningful.} Now, we find that the nullcline~\eqref{eq:4.13b} can be {recast} as a smooth function of $u$, which is positive  and {possesses} critical points for $u>a/b$ as $v^\prime(u)=0$ is held. In so doing, we get that $-2a+b(u-Ku^3)=0$ must be satisfied. This yields to, as a consequence of the Descartes' rule of signs, existence of two positive roots such that $a/b<u_1^\star<u_2^\star$ as $a,b,K>0$, where the lower value corresponds to a local minimum and the greater one to a local maximum. Thus, from both expressions in~\eqref{eq:4.13}, we get that there are at most three positive steady-states. This is schematically illustrated in~Fig.~\ref{Fig:4.2} for a system defined later in~\eqref{eq:4.11} which is equivalent to~\eqref{eq:4.10} in the interior of the first quadrant and is well-defined along the axis $v=0$. These sketches provide sensitive evidence that two saddle-node bifurcations may occur as two steady-states are created by varying appropriate parameters; see panels~(a) and~(c), where the nullclines are tangent. 

\subsection{Local stability  of positive steady-states}
\label{sec:LinearSt1}

We now pay attention to the local stability of positive steady-states~$(u_*,v_*)$ in~\eqref{eq:4.10}. The Jacobian matrix of~\eqref{eq:4.10} at~$(u_*,v_*)$ is
\begin{align}\label{Jac:1}
\mathbf{J} = \left(\begin{array}{cc}
\dfrac{2u_*}{v_*\left(1+Ku_*^2\right)^2}-b & -\dfrac{u_*^2}{\left(1+Ku_*^2\right)v_*^2} \\
2u_*  & -1 \\
\end{array} \right)\,,
\end{align}
where, by setting $\mathcal{J}:=2u_*\left/\left[v_*\left(1+Ku_*^2\right)^2\right]\right.$, the trace is given by $\text{tra}(\mathbf{J}) = \mathcal{J}-b-1$, and the determinant by $\text{det}(\mathbf{J}) = b-\left(1-\left.u_*^2\right/\left[v_*\left(1+Ku_*^2\right)\right]\right)\mathcal{J}$.
Notice that the entries  $J_{12}$, and  $J_{22}$ of~\eqref{Jac:1} are negative, while $J_{21}$, is positive. Hence, local stability features depend on whether the first entry $J_{11}= \mathcal{J}-b$ is greater or less than one as follows from Hartman--Grobman theorem~\cite{guckenheimer}. This is summarised as follows: 

\begin{proposition}\label{teolocalstability}
Let be $\mathcal{J}:=2u_*\left/\left[v_*\left(1+Ku_*^2\right)^2\right]\right.$.
The local stability of a positive steady-state $(u_*,v_*)$ of~\eqref{eq:4.10} is as follows:
\begin{enumerate}
\item It is  locally asymptotically stable  if $\mathcal{J}< b+1$ and $b>\left(1-\left.u_*^2\right/\left[v_*\left(1+Ku_*^2\right)\right]\right)\mathcal{J}$.
\item It is a repeller  if $\mathcal{J}> b+1$ and $b>\left(1-\left.u_*^2\right/\left[v_*\left(1+Ku_*^2\right)\right]\right)\mathcal{J}$.
\item It is a saddle point if $b<\left(1-\left.u_*^2\right/\left[v_*\left(1+Ku_*^2\right)\right]\right)\mathcal{J}$.
\end{enumerate}
\end{proposition}
{We now analyse the dynamical behaviour in the neighbourhood of the origin.}

\subsection{Dynamics near the origin}
\label{sec:origin}

 Despite vector field \eqref{eq:4.10} being defined in $\mathcal{D} = \{(u,v)\in\mathbb{R}^2 | \ u \geq 0, \, v >0 \}$, a solution with $v=0$ corresponds to absence of static bacteria, which is a feasible  biological scenario. In consequence, similarly to~\eqref{eq:s0} we consider a map $t\mapsto  t / v$ to get
\begin{equation}\label{eq:4.11}
Y_0:\left\{\begin{array}{rl}
 \dot{u} = \tilde{f}(u,v)\,, & \tilde{f}(u,v)=\dfrac{u^2}{1+Ku^2}+av-buv\,,\\
 \dot{v} = \tilde{g}(u,v)\,, & \tilde{g}(u,v)=u^2v+aev-v^2\,, 
\end{array}\right.
\end{equation}
which is topologically equivalent to (\ref{eq:4.10}) and can be continually extended to the boundary axis $v=0$.

\begin{proposition}\label{originstability}
The origin of~\eqref{eq:4.11} is a repeller.
\end{proposition}

\begin{proof}
 Since the eigenvalues of the Jacobian matrix at the origin are $\lambda_1=0$ and $\lambda_2=ae>0$, with associated eigenvectors $v_1=(1,0)^t$ and $v_2=(1/e,1)^t$, respectively, it follows that there exists a centre manifold
$$W^c=\{(u,v):\ v=h(u):=a_2u^2+O(u^3),\ |u|<\delta\},$$
for $\delta>0$ sufficiently small~\cite{guckenheimer}.  Since $W^c$ is invariant one must have
$$\tilde{g}(u,h(u))\equiv h'(u)\tilde{f}(u,h(u)),$$
for every $(u,v)\in W^c$. It follows from this identity that $a_2=0$ and $h(u)=O(u^3)$. Hence, the dynamics of~\eqref{eq:4.11} restricted to $W^c$ is equivalent to $\dot x=x^2+O(x^3).$ Therefore, the origin of~\eqref{eq:4.11} is a repeller.
\end{proof}

\section{Bifurcation analysis  in the 2D system}
\label{sec:bifanal}

 We now perform a numerical exploration with {\sc MatCont}~\cite{book:Govaerts2011} in order to identify the most relevant bifurcations when parameters $e$ or $b$ are  varied slowly: Upon varying  $b$, different decaying bacteria rate scenarios are explored; meanwhile, variations of $e$ reveal the impact of having several response rates as well as different production rates. We focus our analysis on the way bacteria decaying rates coordinate the emergence of oscillations. 

Figure~\ref{Fig:4.5} depicts bifurcation diagrams of system~\eqref{eq:4.10} for parameters $e$ and $b$. In the left-hand panel,  two saddle-node~(SN) points are identified at $e_1$ and $e_2$ (see Fig.~\ref{Fig:4.5}(a)), where the system {possesses} only one unstable or stable steady-state for $e<e_1$ and $e>e_2$, respectively, { and three coexisting equilibrium points for $e_1<e<e_2$. Note that the case $e=e_1$ corresponds to the nullcline arrangement in Fig.~\ref{Fig:4.2}(a); the scenario with $e_1<e<e_2$ relates to Fig.~\ref{Fig:4.2}(b); the case $e=e_2$ corresponds to Fig.~\ref{Fig:4.2}(c); and $e>e_2$ is associated with Fig.~\ref{Fig:4.2}(d)}. On the other hand, in Fig.~\ref{Fig:4.5}(b), {after setting $e=3$} two SN points are shown at $b_1$ and $b_2$ in addition to a Hopf point (denoted by H, for simplicity) at~$b_3$. {In both panels of Fig.~\ref{Fig:4.5}, the black curves represent asymptotically stable equilibria, the discontinuous branch between the two SN points is a saddle point, and the upper dashed curve is a repeller.} This Hopf bifurcation  is  supercritical as the first numerically computed Lyapunov coefficient is negative. {Hence, for $b_3<b<b_4$ a stable limit cycle arises  and coexists with the stable equilibrium at the lower branch. This paves the way for the emergence of synchronising dynamics under suitable initial conditions, as a result of the limit-cycle branch that originates from the H point and terminates at the lightly dashed saddle branch at $b_4$, suggesting the potential occurrence of homoclinic orbits.}

\begin{table}[t]
\centering
  \begin{tabular}{ |c|c|c|c|c|c|c|c|}
    \hline
    $a$ &  $b$ & $e$ &$K$ \\ \hline
    $[0,0.15]$ & $[0.5,3]$ & $[0,5]$& $0.043$ \\ \hline
  \end{tabular}
  \caption{Parameter values used in section~\ref{sec:bifanal}.}
  \label{table:Ap1}
\end{table}

\begin{figure}[t!]
    \centering
    \includegraphics[scale=1]{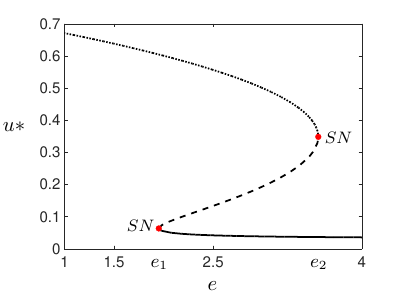}
    \includegraphics[scale=1]{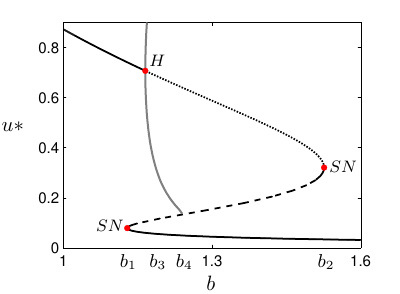} \\
     $(a)$ \hspace{6.5cm} $(b)$
\caption{Bifurcation diagrams for bifurcation parameters $e$ and $b$. {Dashed and solid curves correspond to unstable and stable steady-states, respectively, where heavily dashed branches gather repeller steady-states while saddle steady-states are collected by slightly dashed branches, and solid grey branch amass stable limit-cycles. (a)~With $b=1.4$ and $a=0.042$ fixed, we identify two saddle-node (SN) points located at $e_1\approx1.948551$ and $e_2\approx3.497398$}. (b)~Upon fixing $e=3$ { and $a=0.042$}, we find two SN points at $b_1\approx1.128239$ and $b_2\approx1.512641$ as well as a Hopf (H) point  at $b_3=1.168852$ { which terminates at $b_4\approx1.239672$. Parameter value for $K=0.043$.}} 
\label{Fig:4.5}
\end{figure}

{A two-parameter continuation is conducted following Fig.~\ref{Fig:4.5} to investigate how the bifurcations evolve with the simultaneous variation of both parameters. From the SN points in Fig.~\ref{Fig:4.5}(a) as parameters  $e$ and $a$ are slowly varied reveals a cusp point (CP) and a Bogdanov--Takens point (BT); see Fig.~\ref{Fig:4.6}. 
Additionally, the two-parameter continuation with respect to both $b$ and $e$, starting from the H point, reveals a Bogdanov--Takens bifurcation as the organising centre of the dynamics, as shown in Fig.~\ref{Fig:4.8}}. 
Upon taking into account Proposition~\ref{originstability}, the parameter space of interest for $b$ and $e$ can be divided into the following regions: 

\begin{figure}[t!]
    \centering
    \includegraphics[scale=.35]{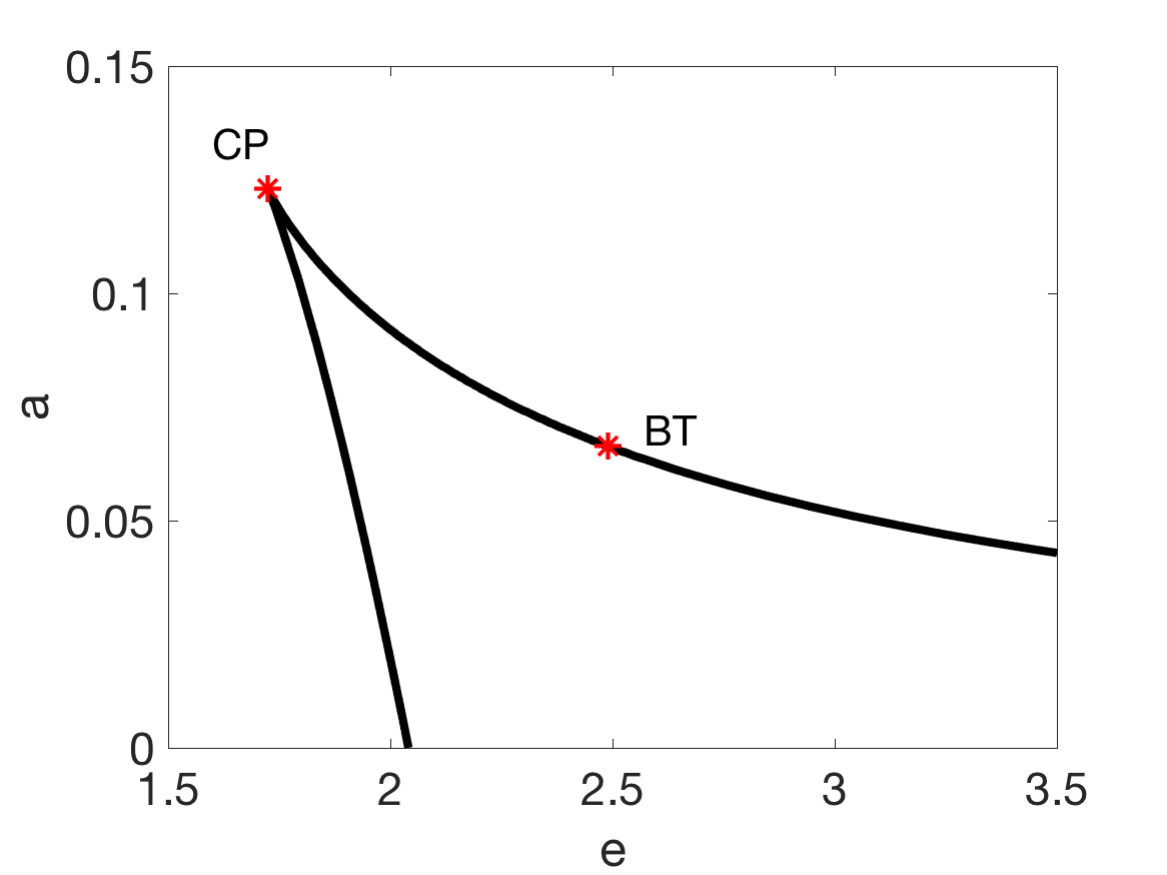} 
\caption{Two parameter continuation. The solid curve corresponds to the SN locus. A CP point is located at $(e,a)\approx(1.723644,0.123118)$ and a BT point at $(e,a)\approx(2.490826,0.066303)$. {Other parameters values: $b=1.4$ and $K=0.043$.}}
\label{Fig:4.6}
\end{figure}

\begin{figure}[t!]
    \centering
    \includegraphics[scale=.5]{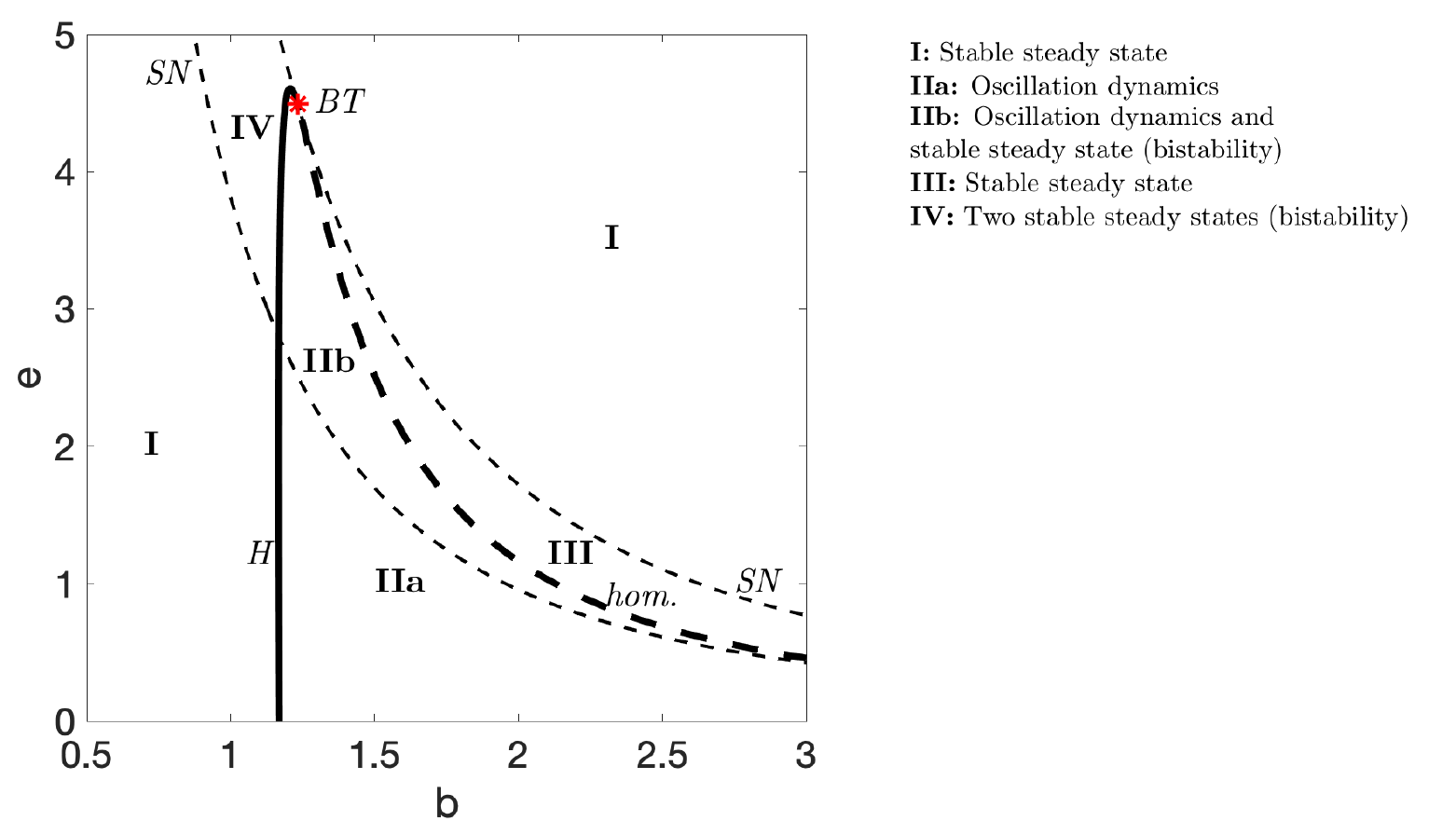}
\caption{Two-parameter continuation for parameters $b$ and $e$. Four possible scenarios are categorised as follows. \textbf{I}:~\textsl{Stable steady-state}, the population converges to a positive population density; \textbf{IIa}:~\textsl{Oscillation dynamics}, the bacteria approach a periodic solution; \textbf{IIb}:~\textsl{Oscillation dynamics and stable steady-state (bistability)}, this region shows oscillation dynamics and a stable node. The bold dashed curve captures homoclinic nonlinear transitions (\textsl{hom}); \textbf{III}:~\textsl{Stable steady-state}, no periodic solutions arise as orbits asymptotically converge to a steady state; \textbf{IV}:~\textsl{Two stable steady-states (bistability)}, there are two stable steady-states, a node and a focus, which are separated by a saddle point. The remaining parameters are set to: $a=0.042$ and $K=0.043$.
}
\label{Fig:4.8}
\end{figure}

{
\begin{enumerate*}[label=(\roman*)]
\item[{\bf (I)}] Stable steady-state: this region, though disconnected, exhibits a stable steady state, while the origin remains unstable. For values of $b$ close to zero, the positive steady state is a stable node. As~$b$ increases slightly, this state transitions into a stable focus, driven by a shift from negative real eigenvalues to complex eigenvalues with negative real parts.
\item[{\bf (II)}] Oscillation dynamics: the emergence of positive, stable periodic solutions plays a critical role in shaping the synchronised dynamical behaviour of bacterial populations. This phenomenon may stem from strong inter-bacterial communication. Notably, the lower SN dashed curve divides this region into two distinct subregions: in subregion~{\bf (IIa)}, self-sustained oscillations dominate, whereas in subregion~{\bf (IIb)} (bistability), a stable steady state coexists with a limit cycle.
\item[{\bf (III)}] Stable steady-state: no limit cycle is observed on the right-hand bold dashed curve associated with the homoclinic bifurcation originating from the BT point. 
\item[{\bf (IV)}] Two stable steady-states (bistability): this region contains three steady states: a stable node and a stable focus, separated by a saddle point. Consequently, the population converges to one of these two stable states depending on the initial conditions.
\end{enumerate*}
}

Notice that Fig.~\ref{Fig:4.8} also shows that the larger the value of $e$, the smaller the range of values of $b$ that can sustain oscillations. In other words, as $e$ surpasses a certain value no synchronising dynamics are possible, no matter the value of $b$. Parameter $e$ plays a bacteria production role in the model, it is directly proportional to the fraction of static bacteria production rate and inversely proportional to the constant autoinducer concentration. This could indicate that the oscillatory dynamics of the population is dependent on the interaction between the decaying and production rates of the bacteria or the decaying rates and the autoinducer concentration. The latter scenario will be explored throughout section \ref{sec:bifanal3D}.

\subsection{Hysteresis and synchronisation}

Now, as we have found such key dynamical events, which are characterised by parameters~$b$ and $e$, {related to autoinducers concentration and bacterial decaying rates, as can be seen in~\eqref{eq:4.10c},} we illustrate these transitions by considering a time-dependent varying variable $b=b(t)$.

\subsubsection{Linearly varying $\boldsymbol{b}$}\label{Sec:2.Bif1a}

The model shows hysteresis as $b$ is varied linearly forth and backwards. For instance, in Fig.~\ref{Fig:4.11a}(a) we fix $e=1.0$ and continuously vary linearly $b=1.0$ up to $b=3.0$. As can be seen, the solution decays in oscillatory fashion  for values of $b$ small enough { (region {\bf I})}; then, as $b$ crosses a corresponding Hopf bifurcation value, the solution amplitude starts increasing approaching a stable limit cycle { (region {\bf II})}; then it finally approaches a stable steady-state after the homoclinic bifurcation { (regions {\bf III} and {\bf I})}. In Fig.~\ref{Fig:4.11a}(b), parameter~$b$ is traversed backwards, which yields a solution that remains in a stable node steady-state throughout the first two regions it crosses { (regions {\bf I} and {\bf III})}. As it reaches the Oscillation dynamics region, the solution spikes and approaches a limit cycle { (region {\bf IIa,b})}. Finally, the solution begins to show a decaying amplitude oscillatory behaviour { (region {\bf I})}. 

 \begin{figure}[t]
  \centering
  \includegraphics[scale=.3]{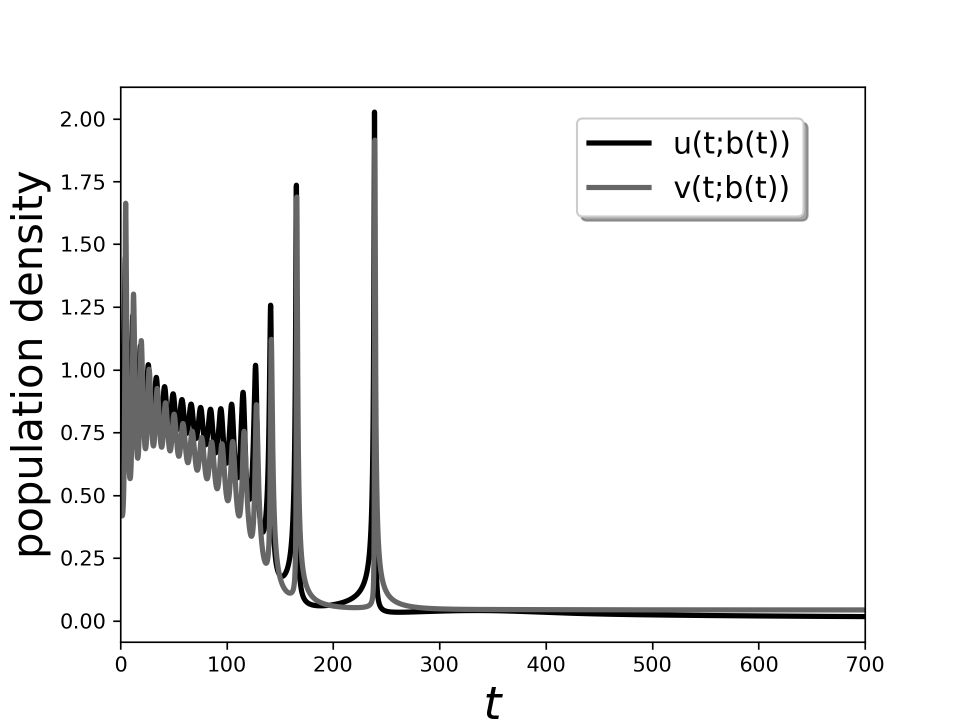}
  \includegraphics[scale=.3]{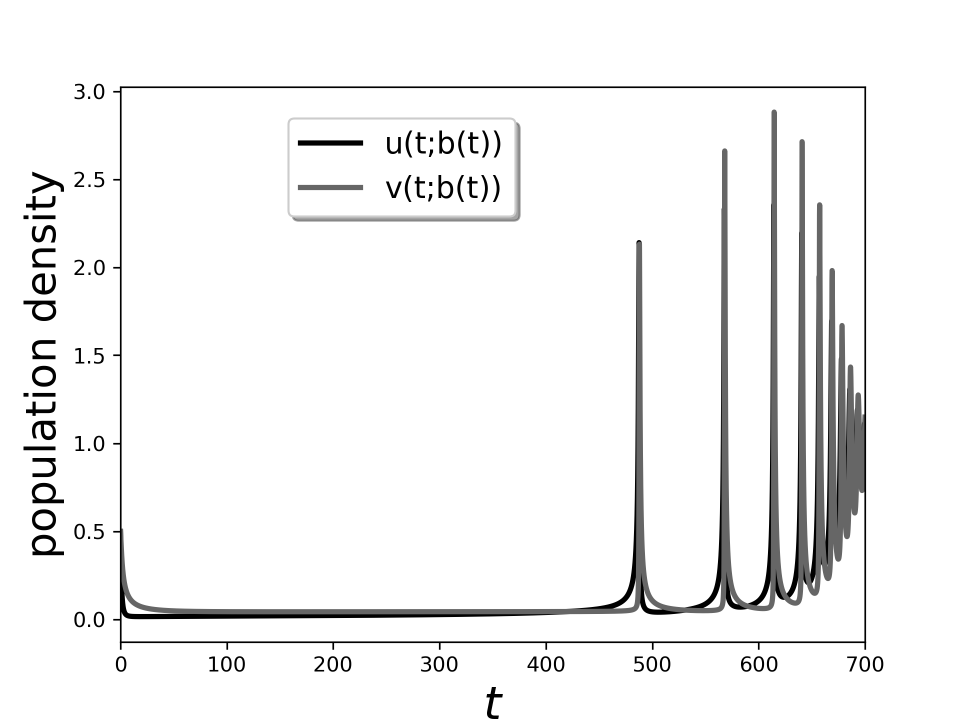} 
  \includegraphics[scale=.3]{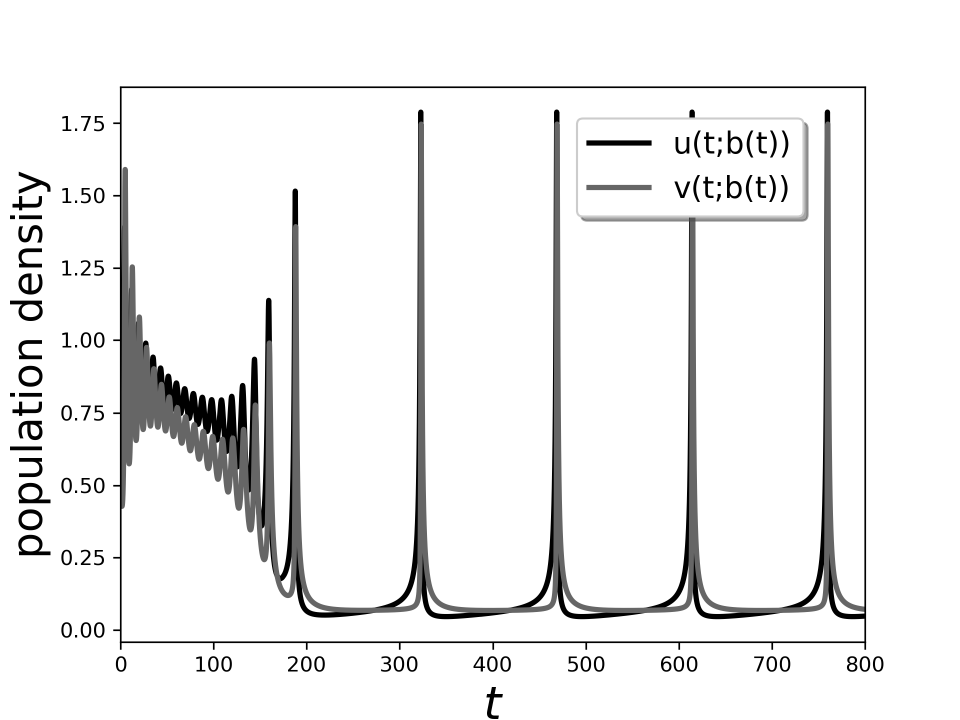}\\
   $(a)$ \hspace{4.2cm} $(b)$ \hspace{4.2cm} $(c)$
\caption{A solution sample for time-dependent parameter $b$: (a)~a solution with $b$ linearly ranging from $b=1$ to $b=3$ with $e=1$ fixed; (b)~a solution with parameter $b$ going backwards from $b=3$ to $b=1$  with $e=1$ fixed; (c)~parameter $b$ linearly increases from $b=1.0$ to $b=1.5$  with $e=1.5$ fixed. Initial conditions are set to $(u,v)=(0.5,0.5)$. {Other parameters values: $a=0.042$ and $K=0.043$.}}
\label{Fig:4.11a}
\end{figure}

Notice that in Fig.~\ref{Fig:4.11a}(b), the solution's behaviour is rather unsimilar to the one when~$b$ increases (see~Fig.~\ref{Fig:4.11a}(a)), which indicates the existence of hysteretic-like features in the system as $b$ is varied forth and backwards. {In other words, the solutions depart from and return to the same initial state, exploring different dynamical regimes depending on the parameter pathways. These solutions exhibit both oscillatory and non-oscillatory behaviours with distinct characteristics, suggesting a robust feature of the system's interaction. This resilience may be interpreted biologically as the bacteria population ability to preserve functionality despite external and internal perturbations, as highlighted in, for example,~\cite{article:Kitano2004}.} Varying $b$ from $b=1.0$ to $b=3.0$ and backwards illustrates how the dynamics of the solution change for different decaying bacteria rate ratios for a constant value of~$e$. When $b$ values are close to one, the decaying rates of constant and static bacteria are similar. As $b$ grows, the difference between the decaying rates of the bacteria increases.

On the other hand, observe that in Fig.~\ref{Fig:4.11a}(c), parameter $e=1.5$  is kept fixed and $b$ varies from $b=1.0$ up to $b=1.5$ and remains at this value as the solution converges to a periodic solution { (regions {\bf I} and {\bf IIa})}. This periodic behaviour shows that the concentration of motile bacteria slightly rises before the static bacteria concentration and also, its exponential decay is faster than the static population one.

\subsubsection{Periodically varying $\boldsymbol{b}$}\label{Sec:2.Bif1b}

{The bifurcation diagram in Fig.~\ref{Fig:4.8} highlights the critical role of parameter $b$ in the emergence of oscillatory behaviour, reflecting the impact of time-dependent variations in bacterial decay rates.} We therefore proceed to illustrate crucial consequences as~$b$~varies periodically back and forth between $b=1.0$ and $b=3.0$ in a given period of time for a fixed $e$~value. In doing so, we will gain insight on the effect of bacteria decay rates that vary in time. We consider an interval given by $t \, \in \, [0,2000]$ and a periodic decay rate parameter given by $b(t)=A+B\sin(\omega t + \varphi)$, where $A$ and $B$ determine parameter bounds and $\omega$ and $\varphi$ are the frequency and phase at which $b$ varies. Fig.~\ref{Fig:IdaVuelta} shows three scenarios for $b$ varying at different frequencies. These simulations suggest that, in order to observe meaningful dynamics, the value of $\omega$ must be small.
For small~$\omega$,  the solution will indeed show oscillations with a significantly small amplitude, but as the value of $\omega$ increases, {the solution will behave as a constant state intermittently disrupted by spikes over time.} In Fig.~\ref{Fig:IdaVuelta} the solutions for $\omega = 0.010,\, 0.027$ and $0.044$ are depicted. For $\omega= 0.010$, we see that as $b$ starts to increase, the solution oscillates regularly. 
 As $b(t)$ varies periodically the solutions alternate between periodic behaviour, oscillatory decay, and constant. That is, as $b$ increases and  the solution enters the oscillation region in Fig.~\ref{Fig:4.8}, the amplitude of the oscillations increases to remain constant. Notice that, as $b$ starts to decrease, the solution becomes constant until it re-enters the oscillation region, when it spikes. This is characterised by the valley-like windows; see upper panel. When $\omega=0.027$ and $0.044$ we see that, as parameter~$b$ takes decreasing values, the solution spikes as $b$ enters the oscillation region. However, as $\omega$ increases, the stable spiral behaviour lasts less time, until it is no longer present in the solution. These sudden spikes in the solution indicate excitatory dynamics in the system. The previous results convey that the dynamics of the 2D-system~\eqref{eq:4.11} is sensible to fast varying parameter $b$, which suggests that the timescale at which the bacteria decay rates change has a rather direct impact on the possible emergence of synchronisation dynamics. 

\begin{figure}[t!]
  \centering
  \includegraphics[scale=.45]{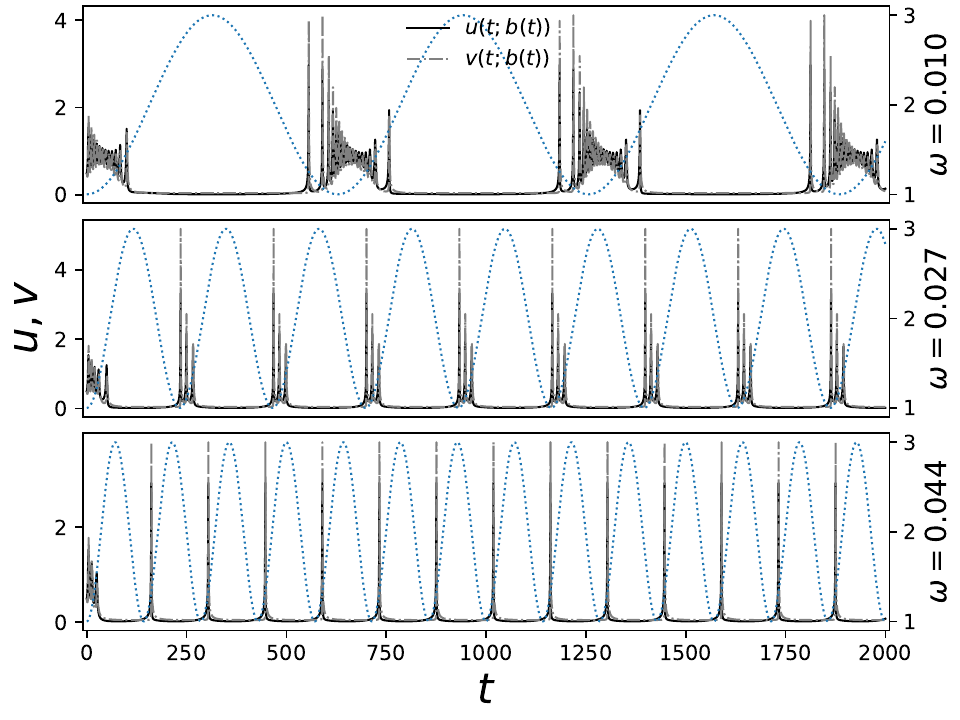}
\caption{A sample solution for $e=1.0$ and $b$ varying between one and three  at  distinguished frequencies. Parameter $b(t)$ is depicted by the dotted curve for  $A= 2$, $B=1$ and $\varphi=3\pi/2$. The solution is shown for three different frequency values as is shown in each panel. {Other parameters values: $a=0.042$ and $K=0.043$.}}
\label{Fig:IdaVuelta}
\end{figure}

\subsection{Andronov--Hopf bifurcation in the 2D system}\label{sec:hopf}

We now proceed to compute the normal forms for the Hopf bifurcation as is crucial for the dynamical shaping for the distinguished parameter values of Table~\ref{table:Ap1}. 

Let $(U,V)\in\mathbb{R}^2_+$ be the positive coordinates of an equilibrium of \eqref{eq:4.10} in the first quadrant. Then the set of equations $f(U,V)=0$ and $g(U,V)=0$
define implicitly a locally invertible transformation given by $\Psi:\Lambda \longrightarrow\mathbb{R}^4_+$,
\begin{flalign}\label{eq:Psi}
(U,V,b,K)\mapsto(a,e,b,K):=&\left(\dfrac{U \big(b (1 + K U^2) V-U\big)}{V + K U^2 V}, \dfrac{(V-U^2) (V + K U^2 V)}{U \big(b (1 + K U^2) V-U\big)}, b,K\right),
\end{flalign}
in $\Lambda:=\{(U,V,b,K)\in\mathbb{R}^4_+:\,b (1 + K U^2) V-U>0,\,\, V-U^2>0\}.$
Then the vector field \eqref{eq:4.10} in parameter space $\Lambda$ has the form:
\begin{equation}\label{hopf1}
\left\{
\begin{aligned}
x' &=	 \dfrac{x^2}{y + K x^2 y} + \dfrac{U \big(b (1 + K U^2) V-U\big)}{V + K U^2 V} - b x ,  \\
y' &=	 x^2 + V -U^2 - y, 
\end{aligned}
\right.
\end{equation}
where, for convenience, we use notation $(x,y)$ to name the state variables, and we denote $x'=\textrm{d}x/\textrm{d}t$, $y'=\textrm{d}x/\textrm{d}t$. System \eqref{hopf1} is $C^{\infty}$-equivalent to \eqref{eq:4.10} in parameter space  $\Lambda$. Moreover, the (positive) equilibrium coordinates appear now as the explicit parameters $(U,V)$. In what follows in this section, we will give conditions such that  \eqref{hopf1} undergoes a generic Hopf bifurcation at $(U,V)$; see \cite{guckenheimer,kuznetsov} for more details.

The Jacobian matrix of \eqref{hopf1}  at $(U,V)$ is
\begin{equation}J(U,V)=\begin{pmatrix} \label{jacUV}
-b+\dfrac{2 U}{(1 + K U^2)^2 V}  & \dfrac{-U^2}{(1 + K U^2) V^2} \\
&\\
2U & -1 \end{pmatrix}.
\end{equation}
The trace $T$ and determinant $D$ of $J(U,V)$ are given by
\begin{equation}\label{trace}
T=T(U,V,b,K)=-1 -b+\dfrac{2 U}{(1 + K U^2)^2 V}
\end{equation}
and
$D=\det\left[J(U,V)\right]= (V + K U^2 V)^{-2} \hat{D}$, respectively, where
\begin{equation}\label{det}
\hat{D}=\hat{D}(U,V,b,K)=2 U^3 + 2 K U^5 - 2 UV + b V^2 + 2 b K U^2 V^2 + b K^2 U^4 V^2.
\end{equation}
Whenever $T=0$ and $D>0$, the eigenvalues of $J(U,V)$ are purely imaginary and non-trivial. Moreover, since we have partial derivative $T_b(U,V,b,K)=-1<0,$
the Hopf bifurcation in \eqref{hopf1} is generically unfolded by parameter $b$. In particular, it follows from there that equation $T(U,V,b,K)=0$ implicitly defines the function $\displaystyle b(U,V,K)=-1 +\dfrac{2 U}{(1 + K U^2)^2 V}$.

We now calculate the  first Lyapunov quantity \cite{guckenheimer,kuznetsov} in order to determine genericity conditions. We follow
the derivation in \cite{guckenheimer} and move the equilibrium $(U,V)$ of the system \eqref{hopf1} to the origin via the translation $x\mapsto x+U$, $y\mapsto y+V$ to obtain the equivalent system
\begin{equation}\label{hopf2}
\left\{
\begin{aligned}
x' &=	 U \left(\dfrac{b - U}{V + K U^2 V}\right) - b (U + x) + \dfrac{(U + x)^2}{V + y + K (U + x)^2 (V + y)},  \\
y' &=	 2 U x + x^2 - y\,.
\end{aligned}
\right.
\end{equation}
In particular, the Jacobian matrix of \eqref{hopf2} at the equilibrium $(0,0)$ coincides with $J(U,V)$ in  \eqref{jacUV}.
 Upon substituting the function $b(U,V,K)$ into $J(U,V)$, we get
\[J_H(U,V)=\begin{pmatrix} 
1  & \dfrac{-U^2}{(1 + K U^2) V^2} \\
&\\
2U & -1 \end{pmatrix},
\]
with $T_H:=\trace\left[J_H(U,V)\right]\equiv0$ and $\det \left[J_H(U,V)\right]=D_H:=D|_{T=0}=\dfrac{2 U^3 - V^2 - K U^2 V^2}{(1 + K U^2) V^2}.$
 
If $D_H>0$, then $\mathbf{v}_1=\left(\dfrac{1}{2U},1\right)^t$ and $\mathbf{v}_2=\left(-\dfrac{w}{2U}, 0\right)^t$ are the generalised eigenvectors of~$J_H(U,V)$,
where $w=\sqrt{D_H}$.
The change of coordinates $\left(x,y\right)^t\mapsto [\mathbf{v}_{1} \, \mathbf{v}_{2}] \left(x, y\right)^t $, where $t$ stands for transpose, allows us to express system (\ref{hopf2}) with $T_H=0$ in the form
\begin{equation}\label{sisparen}
\left(\begin{array}{c}
x'\\ 
y'
\end{array} \right)=\left( \begin{array}{cc}
0 & -w\\ 
w & 0
\end{array} \right)\left(\begin{array}{c}
x \\ 
y
\end{array} \right)+\left(\begin{array}{c}
P(x,y) \\ 
Q(x,y)
\end{array} \right),
\end{equation}
where
\begin{subequations}\label{valorPyQ}
\begin{equation}
P(x,y)=\dfrac{1}{4 U^2}x^2 - \dfrac{w  }{2 U^2} x y+ \dfrac{w^2 }{4 U^2}y^2,
\end{equation}
and
\begin{equation}
\begin{array}{rl}
Q(x,y)=& \dfrac{-16 K U^7 - 8 K^2 U^9 - 2 U V^2 + V^3 + 3 K U^2 V^3 + 3 K^2 U^4 V^3}{4 U^2 (V + K U^2 V)^3 w}x^2\\
& + \dfrac{ K^3 U^6 V^3 + 8 U^5 (-1 + K V) + 2 U^3 V (4 + 3 K V)}{4 U^2 (V + K U^2 V)^3 w}x^2\\
&+\dfrac{-4 K U^5 + 2 U V - V^2 - 3 K U^2 V^2 - 3 K^2 U^4 V^2 - K^3 U^6 V^2 -
 U^3 (4 + 6 K V)}{2 U^2 (1 + K U^2)^3 V^2}xy
\\
&+\dfrac{(-2 U + 6 K U^3 + V + 3 K U^2 V + 3 K^2 U^4 V + K^3 U^6 V) w}{4 U^2 (1 + K U^2)^3 V}y^2\\
&+ \dfrac{ 12 K^2 U^8 + 4 K^3 U^{10} - 4 K U^6 (-3 + K V) + V^2 (1 + 2 K V)}{2U (V + K U^2 V)^4 w}x^3\\
&+    \dfrac{U^4 (4 - 8 K V - 3 K^2 V^2) - 2 U^2 V (2 + K V + K^2 V^2)}{2U (V + K U^2 V)^4 w}x^3\\
&+ \dfrac{4 K^2 U^6 - 2 V (1 + 3 K V) + 2 K U^4 (4 + 3 K V) + 
 U^2 (4 + 4 K V + 6 K^2 V^2)}{2U (1 + K U^2)^4 V^3}x^2y\\
&+\dfrac{(1 - 2 K (U^2 - 3 V) - 3 K^2 (U^4 + 2 U^2 V)) w}{2U (1 + K U^2)^4 V^2}xy^2\\
&+\dfrac{ K (-1 + K U^2) w^2}{U (1 + K U^2)^4 V}y^3+O\left(||(x,y)||^4\right),\\
\end{array}
\end{equation}
\end{subequations}
 in a Taylor expansion near $(x,y)=(0,0).$
 
 System (\ref{sisparen}), equations (\ref{valorPyQ}) and $w=\sqrt{D_H}$ allow us to use the derivation in \cite{guckenheimer} for the direct calculation of the first Lyapunov quantity $L_1$. In so doing, we obtain the following expression:
$$L_1=\dfrac{ 1}{64 w^2 (1 + K U^2)^6 U^4V^5 }\,l_1,$$
where
\begin{flalign}
\begin{array}{rl}
 l_{1}=&-32 K^3 U^{14} + 16 K^4 U^{13} V^2 w^2 + 2 U^2 V^3 \big(2 + 3 K V^2 (w-1) w\big) (1 + w^2) \\
&- 2 U V^4 (1 + w^2)^2 +
   V^5 w ( -1+w - w^2 + w^3) \\ 
& - 
  4 U^3 V^3 \big(-1 + (1 + 6 K V) w^2 + 6 K V w^4\big)+ 
  4 K^3 U^{11} V^2 \big(16 w^2 + K (V + 7 V w^2)\big) \\
&+ 
  U^4 V^2 \big(-24 K V (1 + w^2) - 8 (3 + w^2) + 
     15 K^2 V^3 w (-1 + w - w^2 + w^3)\big)  \\
&+ 
  4 U^6 V \big(12 + 9 K^2 V^2 (1 + w^2) + 4 K V (3 + w^2) + 
     5 K^3 V^4 w (-1 + w - w^2 + w^3)\big) \\
&+ 
  K^2 U^{12} \big(-96 - 48 K V + K^4 V^5 w (-1 + w - w^2 + w^3)\big)\\
  & + 
  6 K U^{10} \big(-16 - 8 K V + K^4 V^5 w (-1 + w - w^2 + w^3)\big)\\
  & + 
  U^8 \big(-32 + 48 K V + 24 K^2 V^2 (3 + w^2)+ 
     15 K^4 V^5 w (-1 + w - w^2 + w^3)\big) \\
    &  - 
  4 U^5 V^2 \big(-4 w^2 - 4 K V (1 + w^2) + 3 K^2 V^2 (-1 + w^4)\big) \\
  & + 
  2 K^2 U^9 V^2 \big(48 w^2 + 8 K (V + 5 V w^2) + 
     3 K^2 V^2 (1 + 6 w^2 + 5 w^4)\big)\\
     &  + 
  8 K U^7 V^2 \big(8 w^2 + 3 K (V + 3 V w^2) + 
     K^2 V^2 (2 + 7 w^2 + 5 w^4)\big).\\
\end{array}
\end{flalign}
Thus, we have obtained the following result.

\begin{theorem}\label{teobifhopf}
Let $(U,V,b,K)\in \Lambda$ be such that $T_H=0$, $D_H>0$ and $l_1\neq0$.
Then (\ref{hopf1}) undergoes a codimension-one Hopf bifurcation at the equilibrium $(U,V)$. In particular, if $l_1<0$ (resp. $l_1>0$), the Hopf bifurcation is supercritical (resp. subcritical), and a stable (resp. unstable) limit cycle bifurcates from $(U,V)$ under suitable parameter variation. 
\end{theorem}

Whenever condition $l_1\neq0$ in Theorem~\ref{teobifhopf} does not hold, the Hopf bifurcation of (\ref{hopf1}) at $(U,V)$ is degenerate. The actual codimension of this singularity ---and the stability of further limit cycles that bifurcate--- is determined by the sign of the so-called second Lyapunov quantity $L_2$; see, for instance,~\cite{kuznetsov}. However, since the transformation \eqref{eq:Psi} is invertible in the parameter set $\Lambda$, numerical evidence suggests that $l_1<0$ for representative parameter values in Table~\ref{table:Ap1} and, hence, the bifurcating limit cycle is stable.
 Furthermore, bifurcation theory ensures that the existence and stability of this stable periodic orbit persist for an open set of parameter values in the region $T(U,V,b,K)>0$, i.e., when the focus at $(U,V)$ is unstable~\cite{guckenheimer,kuznetsov}.

\subsection{Bogdanov--Takens bifurcation in the 2D system}\label{sec:BT}

We give now conditions such that our model undergoes a Bogdanov--Takens bifurcation under suitable parameter variation at a positive equilibrium. We prove the existence of a germ of a BT bifurcation and show that our system (under certain conditions) is locally topologically equi\-va\-lent to a normal form of the BT bifurcation. We refer to~\cite{kuznetsov} and the references therein for the derivation of the genericity and transversality conditions that need to be verified during this proof.

For the sake of clarity, it is convenient to state the dependence of the vector field \eqref{eq:4.10} on parameters $b$ and $K$ explicitly. Hence, throughout this section we denote $X: \mathbb{R}^4_+\longrightarrow\mathbb{R}^2_+$,
\begin{gather}
X(x,y;b,K) =  \bigg( \dfrac{x^2}{(1 + K x^2)y} + \alpha - b x, 
 x^2 + \alpha\epsilon - y  \bigg),
\end{gather}
where we use notation $(x,y)$ for the state variables.
Also, let us denote the Jacobian matrix of $X$ with respect to the variables $(x,y)$ as
$$\dfrac{\partial X}{\partial(x,y)}(x,y;b,K)=\left(%
\begin{array}{cc}
 -\dfrac{-2 x + b y + 2 b K x^2 y + b K^2 x^4 y}{(1 + K x^2)^2 y} & -\dfrac{x^2}{(1 + K x^2) y^2} \\
 2x & -1 \\
\end{array}%
\right)
.$$

{\em Step 1.} We verify that the system may exhibit a singularity with a double zero eigenvalue and geometric multiplicity one. 

\begin{lemma}\label{lem:BT1}
Consider the set
$$\begin{array}{rcl}
\Omega_{BT}&=&\big\{(x,y)\in\mathbb{R}^2_+:\, -2 x^5 + y^3>0, 2 x^3 - y^2>0, -2 x^5 + y^3- x^3 y >0, \\ 
 && (x^2 - y)(2 x^5 + x^3 y - y^3)>0, 4 x^5 - x^3 y - 6 x^2 y^2 + 4 y^3\neq0\big\}.
\end{array}$$
Then there is a positive equilibrium $(U,V)\in\Omega_{BT}$ of  \eqref{eq:4.10} and there are positive parameter values $(b^*,K^*)=\left(\dfrac{-2 U^5 + V^3}{2 U^5},\dfrac{2 U^3 - V^2}{U^2 V^2}\right)\in\mathbb{R}^2_+$ such that $\left.\dfrac{\partial X}{\partial(x,y)}\right|_{(U,V;b^*,K^*)}$ is nilpotent. 
\end{lemma}

\begin{proof}

In order to prove this Lemma, we look look for solutions $(x,y,b,K)$ of the algebraic system
\begin{subequations}\label{eq:systemeq-b}
\begin{eqnarray}\label{eq:systemeq-bt1}
\dfrac{x^2}{(1 + K x^2)y} + \alpha - b x&=&0 ,\\\label{eq:systemeq-bt2}
x^2 + \alpha\varepsilon - y &=&0,\\\label{eq:systemeq-bt3}
\trace\left[\dfrac{\partial X}{\partial(x,y)}(x,y;b,K)\right]&=&0,\\ \label{eq:systemeq-bt4}
\det \left[\dfrac{\partial X}{\partial(x,y)}(x,y;b,K)\right]&=&0,
\end{eqnarray}
\end{subequations}\label{eq:systemeq}
with $(x,y)\in\Omega_{BT}$, $b>0$ and $K>0$.
Solving for $b$ and $K$ in \eqref{eq:systemeq-bt3}-\eqref{eq:systemeq-bt4} leads to 
 $b=\dfrac{-2 x^5 + y^3}{2 x^5}>0$ and $K=\dfrac{2 x^3 - y^2}{x^2 y^2}>0$. 
 Substitution of  $(b,K)$ in \eqref{eq:systemeq-bt1}-\eqref{eq:systemeq-bt2} and solving for $(\alpha,\varepsilon)$ leads to
\begin{gather}\label{eq:systemeq-bt}
\alpha=\dfrac{-2 x^5 + y^3- x^3 y }{2 x^4} \quad \textrm{and} \quad \varepsilon =\dfrac{2 x^4 (x^2 - y)}{2 x^5 + x^3 y - y^3}\,.
\end{gather}
In particular, notice that  $\alpha>0$ and $\varepsilon>0$ since $(x,y)\in\Omega_{BT}.$  The map $(x,y)\mapsto(\alpha,\varepsilon)$ defined by~\eqref{eq:systemeq-bt} has a Jacobian matrix given by 
\begin{equation}\label{eq:jacmap}
	\dfrac{\partial(\alpha,\varepsilon)}{\partial(x,y)}=\left(%
\begin{array}{cc}
 -1 + \dfrac{y}{2 x^2} - \dfrac{2 y^3}{x^5} & -\dfrac{x^3 - 3 y^2}{2 x^4} \\
 \dfrac{2 x^3 (2 x^7 + 5 x^5 y - x^3 y^2 - 6 x^2 y^3 + 4 y^4)}{(2 x^5 + 
  x^3 y - y^3)^2} & \dfrac{-6 x^9 + 6 x^6 y^2 - 4 x^4 y^3}{(2 x^5 + x^3 y - y^3)^2} \\
\end{array}%
\right)\,,
\end{equation}
which determinant is
$$\det \left[\dfrac{\partial(\alpha,\varepsilon)}{\partial(x,y)}\right]=\dfrac{4 x^5 - x^3 y - 6 x^2 y^2 + 4 y^3}{x(2 x^5 + x^3 y -  y^3)}.$$
Therefore, since $\det\left[ \dfrac{\partial(\alpha,\varepsilon)}{\partial(x,y)}\right]\neq0$, the Inverse Function theorem ensures that system~\eqref{eq:systemeq-bt} is locally invertible. Then, the solution of~\eqref{eq:systemeq-b} is given by $(x,y,b,K)=(U,V,b^*,K^*)\in\mathbb{R}^4_+$, where $(U,V)$ is locally defined by the inverse map of~\eqref{eq:systemeq-bt} and $b^*=\dfrac{-2 U^5 + V^3}{2 U^5}$ and $K^*=\dfrac{2 U^3 - V^2}{U^2 V^2}$.

Hence, at $(b,K)=(b^*,K^*)$ the equilibrium $(U,V)$ of the system \eqref{eq:4.10} has a Jacobian matrix given by
\begin{equation}\label{eq:jacBT}
	\dfrac{\partial {X}}{\partial(x,y)}(U,V;b^*,K^*)=\left(%
\begin{array}{cc}
 1 & -\dfrac{1}{2U} \\
 2U & -1 \\
\end{array}%
\right)
\end{equation}
with a double zero eigenvalue. In particular, note that \eqref{eq:jacBT} is not the null matrix. 
The corresponding generalized eigenvectors of \eqref{eq:jacBT} are given by 
\begin{equation} \label{eq:eigenvectors}
	\mathbf{v}_1=\left(\dfrac{1}{2U},1\right)^t, \hspace{5mm} \mathrm{and} \hspace{5mm}\mathbf{v}_2=\left(1,-1 + 2 U\right)^t\,.
\end{equation}
It follows that \eqref{eq:jacBT} is nilpotent and that the double zero eigenvalue has geometric multiplicity one.

\end{proof}

{\em Step 2.}
The next goal is to state the following transversality condition of a Bogdanov--Takens bifurcation.

\begin{lemma}\label{lem:BT2} Let $(U,V,b^*,K^*)$ be as in Lemma~\ref{lem:BT1} and 
consider the map $\displaystyle \Psi:\mathbb{R}^4\rightarrow\mathbb{R}^4$,
$$
	(x,y,b,K)\mapsto\left(\dfrac{x^2}{(1 + K x^2)y} + \alpha - b x, 
 x^2 + \alpha\epsilon - y  ,T,D\right),
$$
where $T$ and $D$ are the trace and determinant of the matrix 
$\dfrac{\partial X}{\partial(x,y)}(x,y;b,K),$
respectively. Then the map $\Psi$ is regular at $(x,y,b,K)=(U,V,b^{\ast},K^{\ast})$.
\end{lemma}

\begin{proof}
The $4\times4$ Jacobian matrix $D\Psi=D\Psi(x,y;b,K)$ of the map $\Psi$ is
\begin{flalign*}
	&D\Psi=\\
	&=\left(%
\begin{array}{ccrc}
 -b - \dfrac{2 K x^3}{(1 + K x^2)^2 y} + \dfrac{2 x}{(1 + K x^2) y} & \dfrac{-x^2}{(1 + K x^2) y^2} & -x & \dfrac{-x^4}{(1 + K x^2)^2 y^2}  \\
 2x & -1 & 0& 0\\
 \dfrac{2 - 6 K x^2}{(1 + K x^2)^3 y} & \dfrac{-2 x}{(y + K x^2 y)^2} & -1 & \dfrac{-4 x^3}{(1 + K x^2)^3 y}  \\ 
 \dfrac{8 K x^4 + 2 K^2 x^6 - 2 y + 6 x^2 (1 + K y)}{(1 + K x^2)^3 y^2} & \dfrac{-4 x^3 - 4 K x^5 + 2 x y}{(1 + K x^2)^2 y^3} & 1 & \dfrac{-2 (x^5 + K x^7 - 2 x^3 y)}{(1 + K x^2)^3 y^2}\\
\end{array}%
\right).
\end{flalign*}

After some calculations, we have $\det \left[D\Psi(x,y,b,K)\right]=\dfrac{2x^5}{(y + K x^2 y)^4}F(x,y,b,K)$ with
\begin{equation} \label{eq:F}
	F(x,y,b,K)=-2 K x^5 - 9 x y + b y^2 + 2 b K x^2 y^2 + b K^2 x^4 y^2 + 
 x^3 (10 + K y).
\end{equation} 
In particular, straightforward substitution and algebraic simplification leads to 
\begin{gather*}
	F(U,V,b^{\ast},K^{\ast})=-\dfrac{2U}{V^2}(4 U^5 - U^3 V - 6 U^2 V^2 + 4 V^3)\neq0\,,
\end{gather*}
since $(U,V)\in\Omega_{BT}$. It follows that $
\det \left[D\Psi(U,V,b^{\ast},K^{\ast})\right]\neq0,
$
which ensures that the map $\Psi$ is regular at $(x,y,C,Q)=(U,V,b^{\ast},K^{\ast})$.

 \qed
\end{proof}

{\em Step 3.} 
We now construct a change of coordinates to transform $X(x,y;b,K)$ into a normal form of the Bogdanov--Takens bifurcation; we refer to~\cite{kuznetsov} once again.

\begin{lemma}\label{lem:BT3} Let $(U,V,b^*,K^*)$ be as in Lemma~\ref{lem:BT1} and consider the following auxiliary expressions:
\begin{gather*}
	G_1=8 U^{10} - 2 U^8 V - 4 U^5 V^3 - 3 U^3 V^4 + 2 V^6\,, \quad G_2=2 U^5 + 3 U^3 V - 2 V^3\,. 
 \end{gather*}
If $G_1\neq0$ and $G_2\neq0$, then there exists a smooth, invertible transformation of coordinates, an orien\-tation-preserving time rescaling, and a re\-pa\-ra\-me\-tri\-za\-tion such that, in a sufficiently small neighbourhood of $(x,y,b,K)=(U,V,b^{\ast},K^{\ast})$,  system \eqref{eq:4.10} is topologically equivalent to a normal form of the codimension-two Bogdanov--Takens bifurcation.
\end{lemma}

\begin{proof}
Let us set up the equilibrium $(U,V)\in\Omega_{BT}$ of \eqref{eq:4.10} to the origin via the translation $x\mapsto x+U$, $y\mapsto y+V$ to obtain the equivalent system
\begin{equation}\label{eq:Y}
Y:\left\{
\begin{aligned}
x' &=	 \alpha  + (x + U)\left( -b+\dfrac{ (x + U)}{(1 + K (x + U)^2) (y + U)}\right),  \\
y' &=	 \alpha\varepsilon + (x + U)^2 - y - V\,.
\end{aligned}
\right.
\end{equation}
In particular, the Jacobian matrix of \eqref{eq:Y} at the equilibrium $(0,0)$ at the bifurcation point $(b^*,K^*)$ coincides with $\dfrac{\partial {X}}{\partial(x,y)}(U,V;b^*,K^*)$ in  \eqref{eq:jacBT}.

  Let $\mathbf{P}=[\mathbf{v}_1, \mathbf{v}_2]$ be the matrix whose columns are $\mathbf{v}_1$ and $\mathbf{v}_2$; see \eqref{eq:eigenvectors}. Next, consider the following change of coordinates:
\begin{equation}\label{eq:P}
 \left(%
\begin{array}{c}
  u\\
  v \\
\end{array}%
\right)
=
\mathbf{P}^{-1}\left(%
\begin{array}{c}
  x\\
  y \\
\end{array}%
\right).
\end{equation}
Then, the vector field given by
$$
	\mathbf{J}=\mathbf{P}^{-1}\circ Y\circ\mathbf{P},
$$
is $\mathcal{C}^{\infty}$-conjugated to $Y$ in \eqref{eq:Y}. 

Taking a Taylor expansion of $\mathbf{J}(u,v;b,K)$ with res\-pect to $(u,v)$ around $(u,v)=(0,0)$ and evaluating at $(b,K)=(b^{\ast},K^{\ast})$, one obtains
$$
\left(%
\begin{array}{c}
  \dot{u}\\
  \dot{v} \\
\end{array}%
\right)
=
\left(%
\begin{array}{cc}
  0 & 1 \\
  0 & 0 \\
\end{array}%
\right)
\left(%
\begin{array}{c}
  u\\
  v \\
\end{array}%
\right)+
\dfrac{1}{2}
\left(%
\begin{array}{c}
 a_{20}u^2+2\,a_{11}uv+a_{02}v^2+O(||(u,v)||^3)\\
   b_{20}u^2+2\,b_{11}uv+b_{02}v^2+O(||(u,v)||^3)\\
\end{array}%
\right),
$$
where we have that: $\displaystyle a_{20}=\dfrac{1}{4U^{10} V}\big( 8 U^{10} - 16 U^{11} + 4 U^9 V - 4 U^5 V^3 + 8 U^6 V^3- 3 U^3 V^4 + 6 U^4 V^4 + 2 V^6 - 4 U V^6\big)$, $\displaystyle b_{20}=\dfrac{1}{4 U^{10} V}\left(8 U^{10} - 2 U^8 V - 4 U^5 V^3 - 3 U^3 V^4 + 2 V^6\right)$ and $\displaystyle b_{11}= \dfrac{1}{2 U^9 V}\left(-4 U^9 + 8 U^{10} - 2 U^8 V + U^4 V^3\right) - \dfrac{1}{2 U^9 V}\left(4 U^5 V^3 + 3 U^3 V^4 - 2 V^6\right)$.

If $b_{20}\neq0$ and $a_{20}+b_{11}\neq0$, then the theory of normal forms for bifurcations~\cite{kuznetsov} ensures that our system fulfills the necessary genericity conditions to undergo a codimension two Bogdanov--Takens bifurcation. 
In particular, condition $G_1\neq0$ ensures that $b_{20}\neq0$. Furthermore, after some algebraic manipulation one obtains $\displaystyle a_{20}+b_{11}= -V^2G_2/4U^{10}$.

In summary, Lemmas~\ref{lem:BT1} and \ref{lem:BT2}, and inequality $G_1G_2\neq0$  ensure that the genericity and transversality conditions of a codimension two Bogdanov--Takens normal form are satisfied. Hence there exists a smooth, invertible transformation of coordinates, an orientation-preserving time rescaling, and a reparametrization such that, in a sufficiently small neighbourhood of $(x,y,\alpha,k)=(U,V,b^{\ast},K^{\ast})$, the system \eqref{eq:4.10} is topologically equivalent to one of the following normal forms of a Bogdanov--Takens bifurcation:
\begin{equation}
   \left\{%
\begin{array}{rcl} \label{eq:nf-bt}
    \dot{\xi}_1 & = &\xi_2, \\
    \dot{\xi}_2 & = & \beta_1+\beta_2 \xi_2 +\xi_2^2\pm \xi_1\xi_2,\\
\end{array}%
\right.
\end{equation}
where the sign of the term $\xi_1\xi_2$ in (\ref{eq:nf-bt}) is determined by the sign of $(a_{20}+b_{11})b_{20}$.
\end{proof}

The next theorem is a straightforward consequence from the findings in Lemmas~\ref{lem:BT1}, \ref{lem:BT2} and \ref{lem:BT3}, as can be seen in~\cite{guckenheimer,kuznetsov}, and it summarises the main result in this section.

\begin{theorem}\label{teo:bt}Let $(U,V,b^*,K^*)$ be as in Lemma~\ref{lem:BT1} and consider the quantities $G_1$ and $G_2$ defined in Lemma~\ref{lem:BT3}.
	Then if $(b,K)=(b^*,K^*)$, system \eqref{eq:4.10} undergoes a codimension-two Bogdanov--Takens bifurcation at $(x,y)=(U,V).$
\end{theorem}

\begin{table}[t]
\centering
  \begin{tabular}{ |c|c|c|c|c|c|c|c|c|}
    \hline
    $\gamma$ & $k_1$ & $k_2$ &  $r$ & $\alpha$ & $\varepsilon$ & $\mu_1$ & $\mu_2$ & $\mu_3$ \\ \hline
    $[0,3.5]$ & $0.07$ & $[0,0.25]$ & $0.4$ & $0.01$ & $0.1$ & $[0,2.5]$ & $[0,2.5]$ & $0.5$ \\ \hline
  \end{tabular}
  \caption{Parameter values used in section~\ref{sec:bifanal3D}.}
  \label{table:Ap2}
\end{table}

\section{Bifurcation and chaos in the 3D system}
\label{sec:bifanal3D}

\begin{figure}[t!]
\begin{center}
\includegraphics[scale=1]{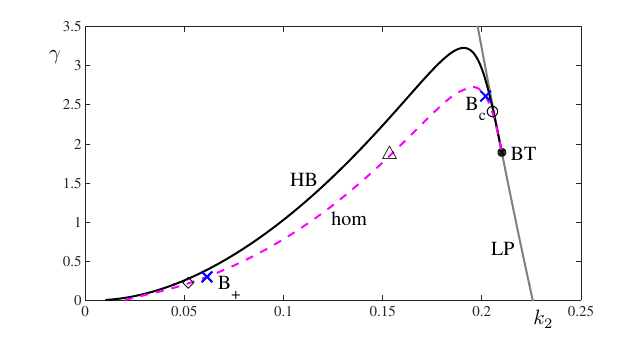}
  \caption{Bifurcation diagram of  \eqref{eq:s0} in the $(k_2,\gamma)$-plane. Shilnikov homoclinic chaos can be found near the homoclinic bifurcation (dashed) curve hom along the lefthand segment from the Belyakov point $B_c$; curves in grey and black correspond to locus of LP and HB points, respectively.}  
\label{fig:k2g}
\end{center}
\end{figure}

{As seen from the re-scaled parameters in~\eqref{eq:4.10c}, parameters $e$, $b$, and $a$ capture the bacterial population's response to autoinducers at constant concentration, as well as the decay rates of both bacterial subpopulations. Based on this, we analyse the effects of slowly varying the static bacteria’s response to autoinducer presence and motile bacteria's saturation parameter, while the autoinducer concentration changes according to its dynamic production. Thus, to understand the emerging global behaviour, we perform a bifurcation analysis of system~\eqref{eq:s0} using {\sc Auto}~\cite{Auto}, allowing both parameters $k_2$ and $\gamma$ to vary slowly. The other parameters remain fixed in their typical values as in Table~\ref{table:Ap2}, except for $\mu_1=0.2,$ and~$\mu_2=0.7$.} Fig.~\ref{fig:k2g} shows the bifurcation diagram in the $(k_2,\gamma)$-plane. There curves of saddle-node (LP) and supercritical Hopf (labelled in this section as HB) bifurcation meet at a BT point. Also, a curve of Shilnikov homoclinic bifurcation (hom) emerges from the point BT. A codimension-two Belyakov point ($B_c$) on the curve hom marks the onset of chaotic dynamics. To the right of~$B_c$, the homoclinic bifurcation is simple; and to the left of $B_c$, the homoclinic bifurcation is chaotic. More concretely, one can find horseshoe dynamics in return maps defined in a neighbourhood of the homoclinic orbit. The suspension of the Smale horsehoes form a hyperbolic invariant chaotic set which contains countably many periodic orbits of saddle-type.
The horseshoe dynamics is robust under small parameter perturbations; hence, the chaotic dynamics persist if the homoclinic connection is broken; see~\cite{guckenheimer, kuznetsov}. A second codimension-two point $B_+$ and also called a Belyakov point, lies on hom for smaller values of both $k_2$ and $\gamma$. At the point $B_+$, the steady-state associated with the homoclinic orbit has repeated stable eigenvalues.
On the  segment of hom to the right of  $B_+$, the homoclinic orbit converges to a saddle-focus (i.e., it has a complex pair of stable eigenvalues); and on the segment to the left of $B_+$, the same equilibrium is a real saddle (i.e., the stable eigenvalues are real). 

\begin{figure}[t!]
\begin{center}
\includegraphics[scale=0.85]{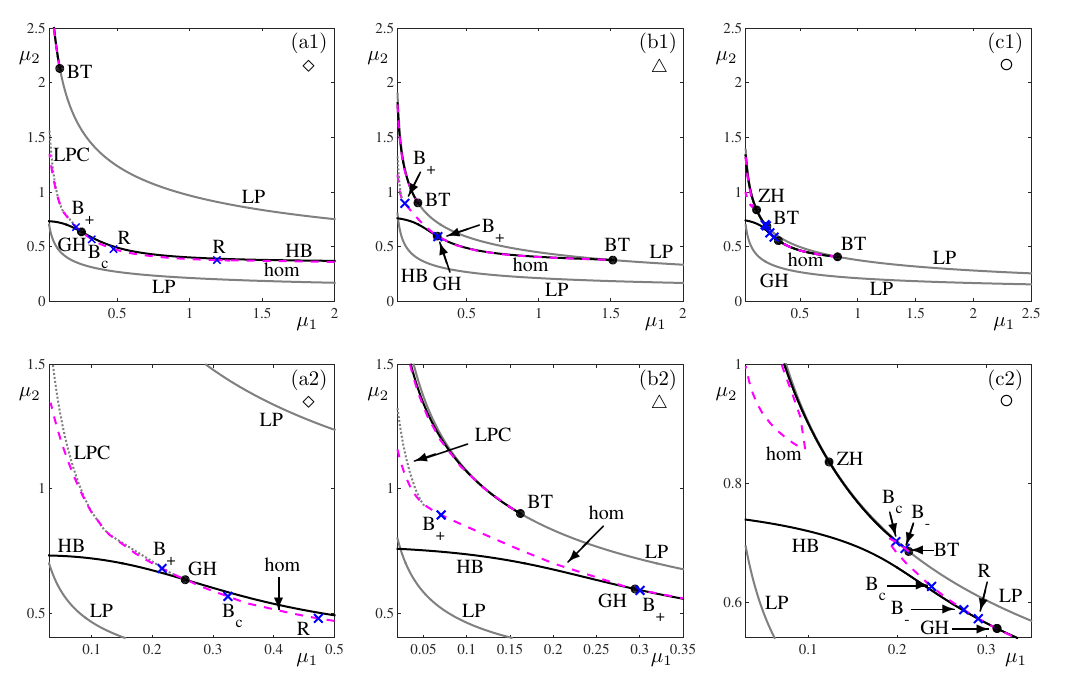}
  \caption{Bifurcation diagram of  \eqref{eq:s0} in the $(\mu_1,\mu_2)$-plane. Bottom row images are enlargements of selected regions in the top figures. Parameter values  $k_2$ and $\gamma$ correspond to those at the points {\large$\diamond$}, $\bigtriangleup$, and {\Large$\circ$} along the homoclinic bifurcation curve {\bf hom} in Fig.~\ref{fig:k2g}. The other parameter values remain fixed as in Fig.~\ref{fig:k2g}.}  
\label{fig:mu1mu2}
\end{center}
\end{figure}

The bifurcation picture in Fig.~\ref{fig:k2g} is just a partial representation of the full complexity one may encounter in this region of parameter space. Indeed, the saddle periodic orbits may also undergo further bifurcations such as period-doubling and torus bifurcations~\cite{guckenheimer, kuznetsov}. Moreover, the presence of the chaotic Shilnikov homoclinic bifurcation and that of the Belyakov points $B_c$ and $B_+$ imply a very complicated structure (not shown) of infinitely many saddle-node and period-doubling bifurcations of periodic orbits as well as of subsidiary $n$-homoclinic orbits~\cite{paguirre,bel80,bel84,gonchenko97,seb05}. Moreover, for each of these subsidiary $n$-homoclinic orbits, the system exhibits countably many horseshoes as in the original homoclinic scenario.

\begin{figure}[t!]
\begin{center}
\includegraphics[scale=1]{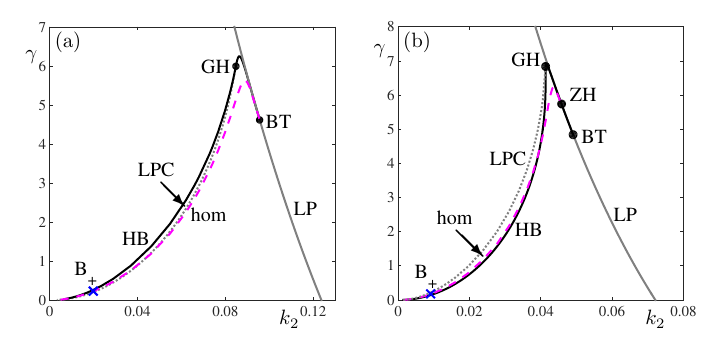}
  \caption{Bifurcation diagrams of  \eqref{eq:s0} in the $(k_2,\gamma)$-plane. Parameter values are  as in Fig.~\ref{fig:k2g}, except $\mu_2=1$ in panel (a), and $\mu_2=1.4$ in panel (b).}  
\label{fig:k2gb}
\end{center}
\end{figure}

The Belyakov points $B_c$ and $B_+$ in Fig.~\ref{fig:k2g} divide the curve hom in three segments. We now fix parameters $(k_2,\gamma)$ at selected points on each of these segments ---labelled as {\large$\diamond$}, $\bigtriangleup$, and {\Large$\circ$}, respectively--- and allow parameters $\mu_1$ and $\mu_2$ to vary. The resulting picture is shown in Fig.~\ref{fig:mu1mu2} in the top row, while suitable enlargements are presented in the bottom row.  While  the resulting bifurcation scenarios shown in Fig.~\ref{fig:mu1mu2} are slightly different to one another, the main ingredients organizing complicated dynamics (such as Belyakov points and homoclinic chaos) remain a common feature in each case. In particular, in the three cases, the bifurcation curves corresponding to the scenario of Fig.~\ref{fig:k2g} are those occurring for lower values of $\mu_2$ in Fig.~\ref{fig:mu1mu2}. The second curve LP  (that which is present for larger values of $\mu_2$) and associated bifurcation phenomena is not present in Fig.~\ref{fig:k2g}.
The Hopf bifurcation curve hom in Fig.~\ref{fig:mu1mu2} is now separated into two segments by a degenerate Hopf point GH from which a curve of saddle-node of periodic orbits (LPC) emerges.
The curve of homoclinic bifurcation hom now contains further codimension two resonant ($R$) and Belyakov ($B_-$) points. While a single stable cycle bifurcates from $B_-$ (and no extra bifurcations occur),  the full picture near the points $R$ include period-doubling bifurcations and a curve of 2-homoclinic bifurcation (not shown) emanating from the codimension-two points. Of particular interest is the case in panels (c) where the curve hom terminates at a BT point located on the second curve LP.  Also, from this BT point a  subcritical Hopf bifurcation emerges. In particular, this HB curve is very close to the saddle-node curve; these two curves meet at a Zero-Hopf bifurcation point (ZH). 
The exact dynamical features which appear in phase space for parameter values near a ZH point depend on higher order terms in a normal form approach, which is beyond the scope of this work. It suffices to say that this codimension-two point is often related to the emergence of invariant tori and chaotic invariant sets~\cite{guckenheimer,kuznetsov}. 
For larger values of $\mu_2$ and lower values of $\mu_1$, another (non chaotic) homoclinic bifurcation curve makes a sharp turn near the point ZH.

Finally, Fig.~\ref{fig:k2gb} shows bifurcation diagrams in the $(k_2,\gamma)$-plane with values of $\mu_2=1$ (in panel~(a)), and $\mu_2=1.4$ (in panel~(b)). The resulting bifurcation diagrams are similar to that in Fig.~\ref{fig:k2g}, but with higher values of $\mu_2$. In  Fig.~\ref{fig:k2gb} the HB curve has a degenerate Hopf point GH, from which a curve of saddle-node of cycles LPC emerges. All in all, upon comparing figures~\ref{fig:k2g} and~\ref{fig:k2gb}, an increase of parameter $\mu_2$ produces extra bifurcations that favours chaotic behaviour. Indeed, the hom curve in  Fig.~\ref{fig:k2g} is now chaotic along its entire length. Also, in Fig.~\ref{fig:k2gb}(b), an increase of $\mu_2$ favours the emergence of a ZH point. However, the bifurcation sets are ``pushed'' towards lower values of $k_2$ and larger values of $\gamma$ as $\mu_2$ is increased; compare the scale of the variables in figures~\ref{fig:k2g} and~\ref{fig:k2gb}. 


\begin{figure}[t]
  \centering
   \includegraphics[scale=.33]{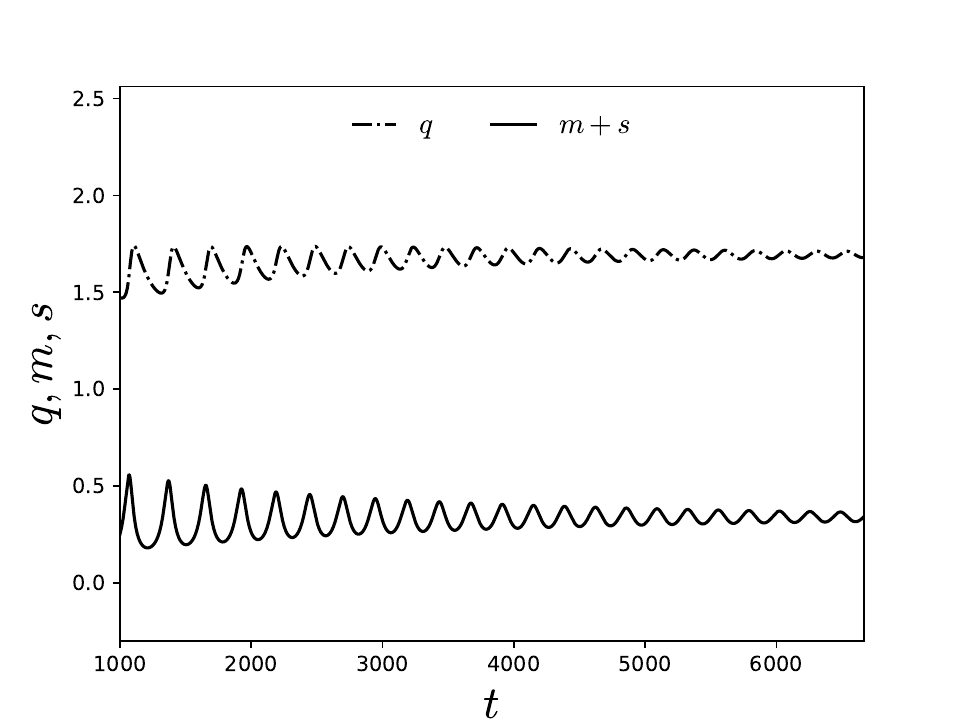}
  \hspace{-0.5cm}
  \includegraphics[scale=.33]{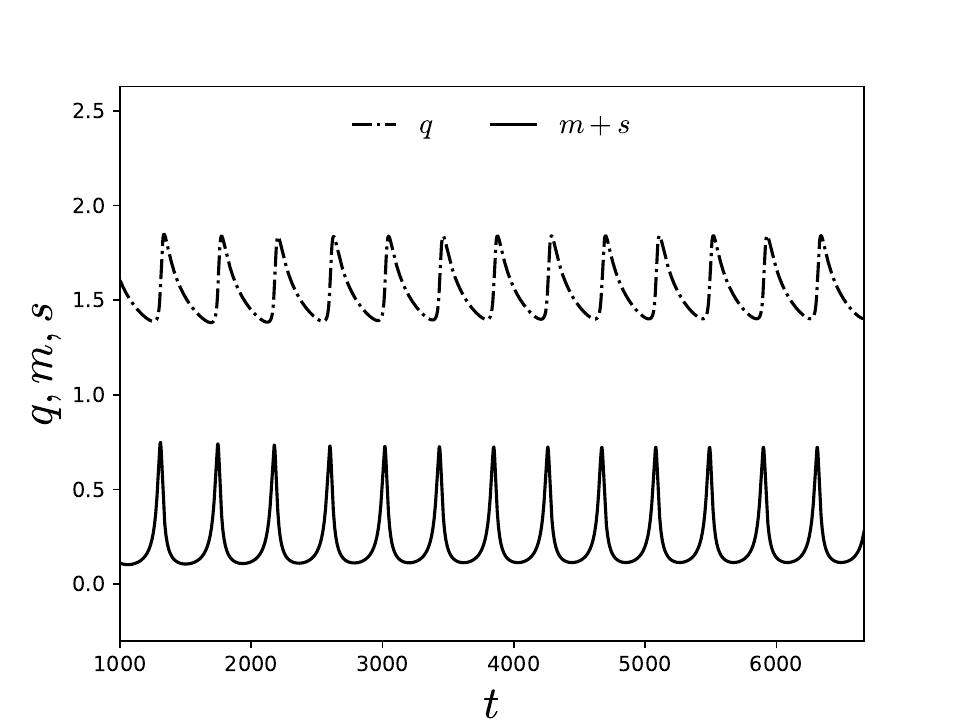} 
  \hspace{-0.5cm}
  \includegraphics[scale=.33]{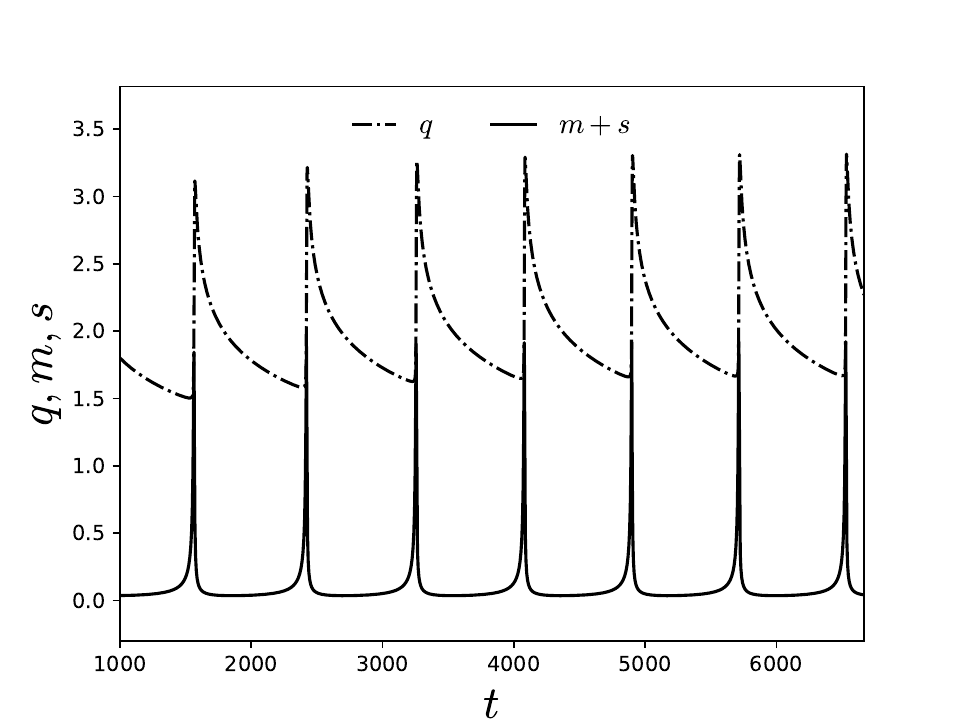}\\
  \includegraphics[scale=.375]{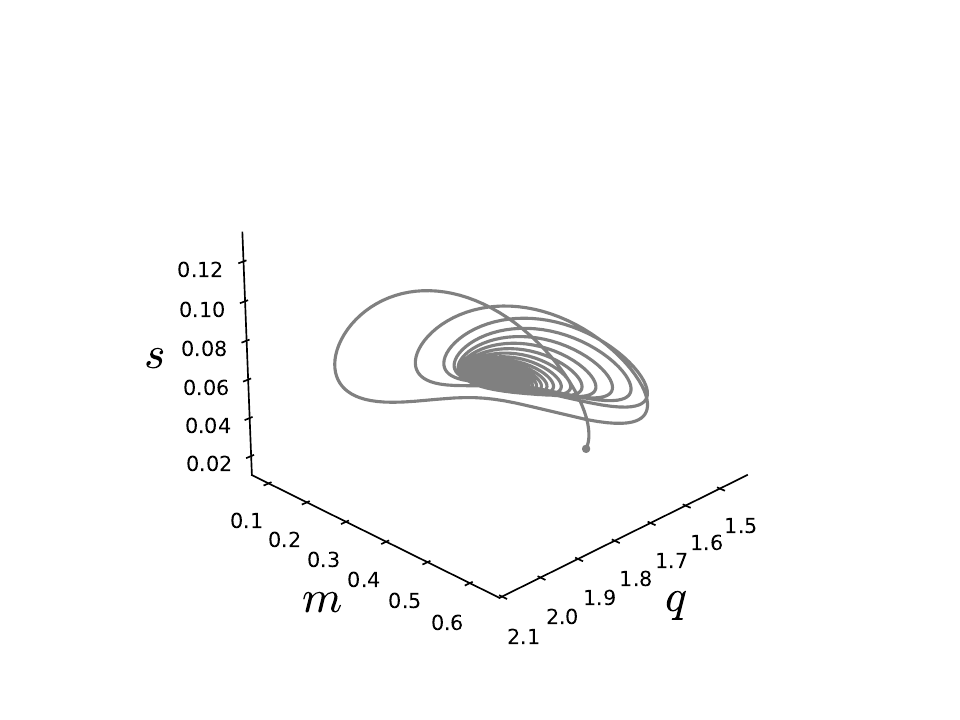}
  \hspace{-1.5cm}
  \includegraphics[scale=.375]{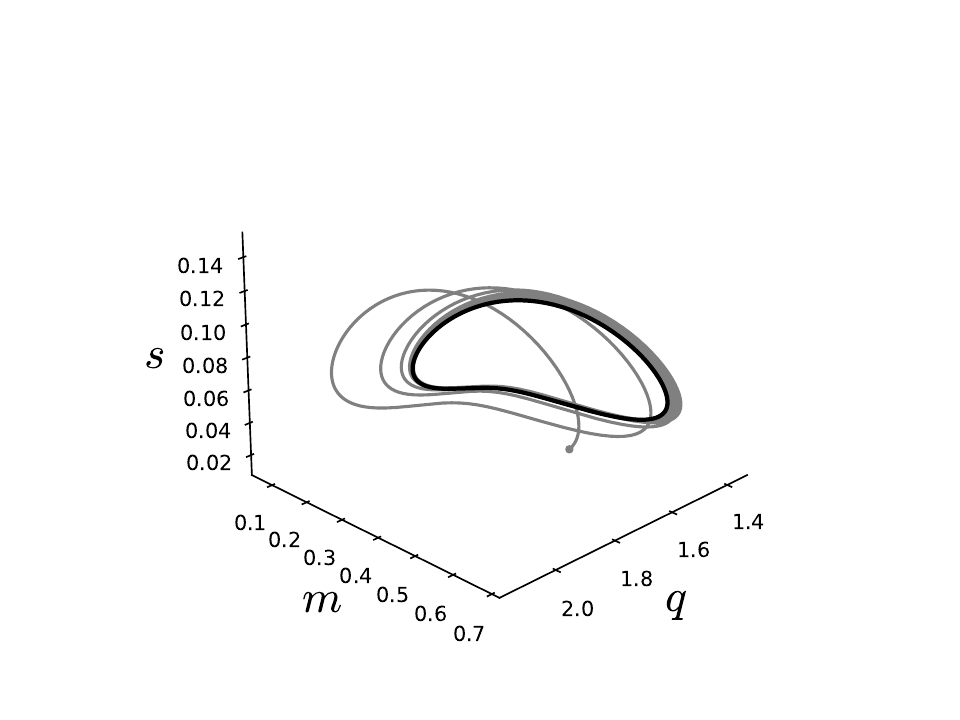} 
  \hspace{-1.5cm}
  \includegraphics[scale=.375]{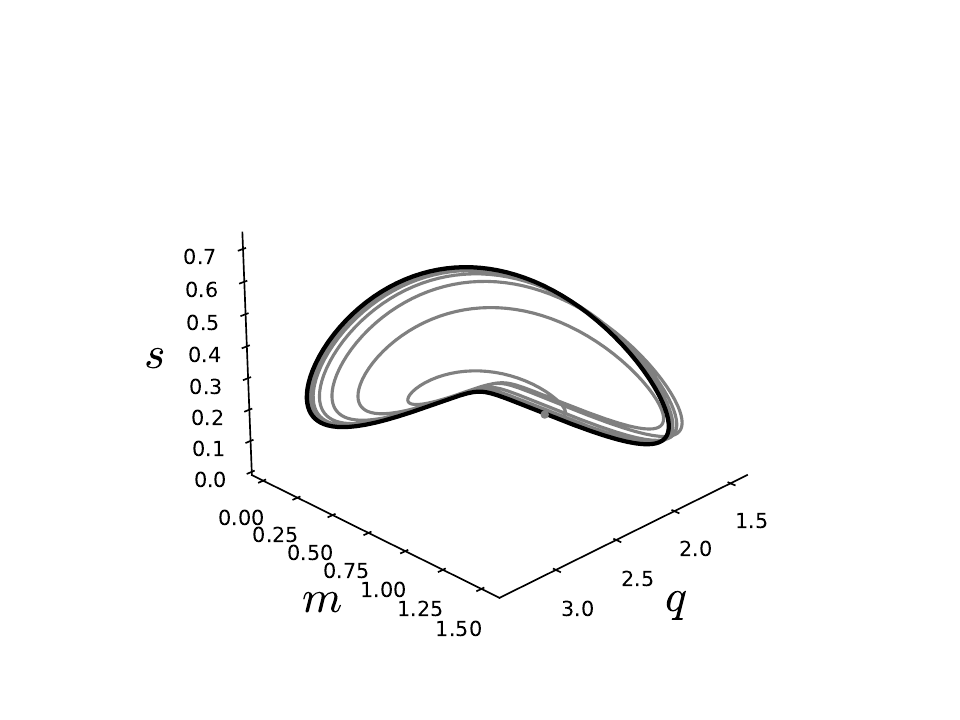}\\
   $(a)$ \hspace{4.5cm} $(b)$ \hspace{4.5cm} $(c)$
\caption{ Samples of time solutions (upper panels) and corresponding phase orbits (lower panels) of system~\eqref{eq:s0}. (a)~Damped oscillatory solution with a stable focus for $\gamma=2.5$; (b)~oscillatory periodic solution with a stable limit-cycle for $\gamma=2$; (c)~spike-like periodic solution near a homoclinic orbit for $\gamma = 1.8$. Initial conditions are chosen close to the saddle point. Parameter values: $k_2 = 0.1535$, $\mu_1 = 0.2$, $\mu_2 = 0.7$, with remaining parameters as listed in Table~\ref{table:Ap2}.}
\label{fig:evol3D}
\end{figure}

{ Sample solutions for autoinducer concentration and bacterial population, along with their corresponding phase orbits, are presented in Fig.~\ref{fig:evol3D} for representative $k_2$ and $\gamma$ parameter values from Fig.~\ref{fig:k2g}. The stable focus shown in column~(a) undergoes a bifurcation into a limit cycle in column~(b) as the parameter $\gamma$ crosses the HB curve. The oscillation period further increases as $\gamma$ approaches the hom curve, as illustrated in column~(c).

This transition is evident in the time solutions displayed in the upper panels, where the crests and valleys of the autoinducer concentration correspond to the peaks and troughs of the bacterial population over time. In column~(a), the amplitudes of both $q$ and $m+s$ decay over time, eventually stabilising at a steady-state. In contrast, columns~(b) and (c) demonstrate a progression in oscillatory behaviour, with the period lengthening and culminating in a spike-like periodic pattern right until the formation of the homoclinic orbit.

The bifurcation diagrams in Fig.~\ref{fig:mu1mu2} and Fig.~\ref{fig:k2gb} reveal that the dynamical richness and complexity of the system are not exclusively determined by the parameters~$k_2$ and~$\gamma$. Decay rate parameters also play a pivotal role in shaping the system’s behaviour. Specifically, chaotic regions and their intricate features are not restricted to the $(k_2,\gamma)$-plane but can also emerge through the gradual variation of other system parameters, highlighting the multidimensional nature of the dynamics. In particular, whenever parameters approach sufficiently close to the hom curve in its chaotic regime, one might expect the stable periodic motion to undergo multiple saddle-node and period doubling phenomena (not shown in the bifurcation diagrams) and, ultimately, encounter regions in phase space where chaos exists.
However, the chaotic solutions cannot be directly accessed through standard time integration from initial conditions due to the saddle-like nature of the hyperbolic set, which prevents it from acting as an attractor. A chaotic trajectory may be reached from the stable manifold of the saddle invariant set. A nearby orbit spends a long transient ``visiting" the chaotic invariant set until it manages to escape and converge to an attracting object.
Furthermore, these chaotic regions in phase space are typically confined to a small neighbourhood surrounding the homoclinic curve, making the identification and characterisation of transient chaos particularly challenging. The detailed exploration of these chaotic sets lies beyond the scope of the present work. For an in-depth theoretical and numerical analysis and further insights into the dynamics of these regions, the reader is referred to~\cite{paguirre}. }

\section{Discussion}
\label{sec:disc}

Analysing the effect of autoinducers on bacterial populations is crucial due to their key role in the QS mechanism. We presented a model that demonstrates the interaction between autoinducers and two subtypes of bacteria. A thorough analysis of this model revealed parameter sets that lead to oscillation dynamics in both the constant autoinducer sub-model and the full three-component model.
The study was performed with a combination of both rigorous normal form analysis and numerical continuation methods.
 Upon analysing the full system, we were able to understand how autoinducer concentrations interact with bacterial populations and influence bacterial communication triggered by these interactions.

The constant autoinducer system \eqref{eq:4.11} shows that parameter $e$ is inversely proportional to the autoinducer concentration $q_0$, indicating that changes in $e$ can impact response rates and production rates on the population. Exploring additional parameters revealed that parameter $b$, related to bacteria decaying rates, leads to a family of limit cycles via Hopf bifurcation. By conducting two-parameter numerical continuation on parameters $b$ and~$e$, the $(b,e)$-plane was divided into four regions based on their dynamical events. It was observed that no synchronising behaviour occurs for high values of $e$, suggesting that oscillatory behaviour may not occur below certain critical shallow threshold autoinducer concentration; this discovery agrees with QS features already recognised in e.~g.~\cite{kaur,Li,taylor}.

Furthermore, as can be seen from Proposition~\ref{teolocalstability}, when $\mathcal{J}=b\left/\left(1-\left. u_*^2\right/\left[v_*\left(1+Ku_*^2\right)\right]\right)\right.$, notice that $\text{tra}(\mathbf{J})=\left[(b+1)u_*^2 - v_*\left(1+Ku_*^2\right)\right]\left/\left[v_*\left(1+Ku_*^2\right)-u_*^2\right]\right.$ and $\text{det}\left(\vectorr{J}\right)=0$, which suggest that other nonlinear events may trigger synchronising behaviour  on the bacterial population not only by means of a Hopf bifurcation. To fully understand such an implication and, in consequence, the impact of autoinducer dynamics on the population, we analysed the 3D system and identified $k_2$ and $\gamma$ as key parameters for synchronising dynamics. Through Hopf bifurcations, self-sustained oscillations were observed to emerge. A two-parameter continuation analysis revealed that these parameters play a role in the emergence of additional oscillatory behaviour and robust synchronisation properties in the population. From section \ref{sec:bifanal3D}, { the identification and mapping of homoclinic bifurcations in the model are crucial for understanding the dynamic behaviour of autoinducer concentration and bacterial population interactions. The sample solutions presented in Fig.~\ref{fig:evol3D} illustrate how the system transitions from a stable focus to oscillatory behaviour through a Hopf bifurcation as the parameter $\gamma$ crosses the HB curve. As $\gamma$ approaches the homoclinic bifurcation, the increasing oscillation period signifies a critical shift in system dynamics, leading to complex behaviours such as spike-like periodic patterns. This progression highlights the importance of recognising homoclinic bifurcations, as they serve as pivotal points that can influence different types of oscillatory behaviour. } 

Moreover, it was observed that a Shilnikov homoclinic bifurcation branch originates from a Bogdanov--Takens point, leading to a Belyakov point $B_c$ that separates simple and chaotic behaviour along the homoclinic curve. Another Belyakov point $B_+$ marks the boundary between real saddle and saddle-focus steady-states. By conducting a two-parameter continuation along three pivotal points on the homoclinic curve, varying parameters $\mu_1$ and $\mu_2$ (equivalent to parameter $b$ in the constant autoinducer system), intricate dynamics were revealed resulting in resonant and additional Belyakov points. Additionally, a zero-Hopf point appears in an Hopf curve intersecting a nearby saddle-node curve. These findings support chaotic dynamics as infinitely many saddle-node, period-doubling orbits and $n$-homoclinic orbits ---accompanied with countably many horseshoes for all $n\in\mathbb{N}$--- come to birth as well as invariant tori and invariant chaotic sets. {In spite of all this rich dynamics, accessing chaotic solutions is hindered by the saddle-like nature of the hyperbolic set. Moreover, these chaotic regions are confined to a small neighbourhood around the homoclinic curve. All this prevents direct access to chaotic orbits through standard simulations.}

The QS model examined in this study depicts synchronised oscillatory dynamics in bacterial populations, potentially exhibiting synchronising characteristics. {This behaviour is not only influenced by the presence of an autoinducer concentration threshold, but also underscores how elevated autoinducer levels drive pronounced peaks in bacterial population dynamics. Understanding these transitions not only enhances our comprehension of the underlying biological processes but also aids in predicting system behaviour under varying conditions, ultimately contributing to more effective modelling and control strategies in microbial systems.} However, to fully understand the global picture of the QS mechanism, migration dynamics and time delays should also be considered. The inclusion of a time delay is crucial as it influences the dynamics of autoinducer reception and emission, indicating that bacteria do not react instantaneously to autoinducers (e.g.~\cite{Mukherjee}). These aspects will be addressed in future research.

\section*{Acknowledgements}  Viviana R-E would like to thank Agencia Nacional de Investigaci\'on y Desarrollo de Chile (Grant: Beca Doctorado Nacional N$^\circ$ 21211263). PA thanks Proyecto Interno UTFSM PI-LIR-24-04 and Proyecto Basal CMM-Universidad de Chile. VFBM thanks the financial support by Asociación Mexicana de Cultura A.C.
\vspace*{-0.23cm}
\bibliographystyle{siamplain}
\bibliography{../QSbiblio_feb2020,QSbiblio24feb}

\end{document}